\documentclass[reqno]{amsart}

\usepackage[T1]{fontenc}
\usepackage{amsmath,amssymb, amsthm}  
\usepackage{bbm} %pogrubiona "jedynka"
\usepackage{hyperref}
\usepackage{mathtools}
\usepackage{enumerate}
\usepackage{graphicx}
\usepackage[inline]{enumitem}
\newcommand{\vectornorm}[1]{\left|\left|#1\right|\right|}

\usepackage{amsaddr}

\newtheorem{thm}{Theorem}[section]

\theoremstyle{definition}

\newcommand{\id}{\operatorname{id}}
\newcommand{\range}{\operatorname{Range}}
\newcommand{\btau}{\overline{\tau}}
\newcommand{\p}{\mathbbm{p}}
\newcommand{\Np}{\mathbf{N}_{\mathbbm{p}}}
\newcommand{\Npw}{\mathbf{N}_{\p_{\Lambda}}}
\newcommand{\m}{\mathbf{m}}
\newcommand{\Ms}{\mathcal{M}_{sig}}
\newcommand{\Mp}{\mathcal{M}_{prob}}
\newcommand{\M}{\mathcal{M}_{+}}
\newcommand{\B}{\mathcal{B}}
\newcommand{\R}{\mathbb{R}}
\renewcommand{\L}{\mathcal{L}}
\newcommand{\n}{\mathbb{N}}
\newcommand{\pr}{\operatorname{Prob}}
\newcommand{\ew}{\mathbb{E}}
\newcommand{\supp}{\operatorname{supp\,}}
\newcommand{\card}{\operatorname{card\,}}
\newcommand{\sbullet}{\raise .5ex\hbox{\tiny$\bullet$}}

\newcommand{\olambda}{\overline{\lambda}}
\newcommand{\ulambda}{\underline{\lambda}}

\newcommand{\norma}[1]{\vectornorm{#1}}
\newcommand{\<}{\left\langle}  
\renewcommand {\>}{\right\rangle}  

\numberwithin{equation}{section}
\mathtoolsset{showonlyrefs} % powoduje ¿e bêd¹ wyœwietlane tylko te numery wzorow do ktorych sie odwo³ujesz
\begin{document}

\title{\bf Exponential ergodicity of some Markov dynamical system with application to a Poisson driven stochastic differential equation}

\author{Dawid Czapla, Joanna Kubieniec}
\address{Institute of Mathematics, University of Silesia in Katowice,\\ 
Bankowa 14, 40-007 Katowice, Poland}
\email{dawid.czapla@us.edu.pl,\,\, jkubieniec@math.us.edu.pl}
\subjclass[2010]{Primary:  37A30, 60J05, 60H10, Secondary: 37H10} 
\keywords{Markov operator, piecewise-deterministic Markov process, asymptotic stability, invariant measure, exponential ergodicity, Poisson driven equation}
\begin{abstract}
We are concerned with the asymptotics of the Markov chain given by the post-jump locations of a certain piecewise-deterministic Markov process with a state-dependent jump intensity. We provide sufficient conditions for such a model to possess a unique invariant distribution, which is exponentially attracting in the dual bounded Lipschitz distance. Having established this, we generalise a result of J. Kazak on the jump process defined by a Poisson driven stochastic differential equation with a solution-dependent intensity of perturbations.
\end{abstract}

\maketitle

\section*{Introduction} 
The starting point of our study is a piecewise-deterministic Markov process (PDMP), which arises from a random dynamical system defined in a manner similar to that in \cite{b:horbacz_diss, b:horbacz_cont, b:dawid, b:horbacz_2006, b:horbacz_sl}. Such a system, say $\{Y(t)\}_{t\geq 0}$, evolves through random jumps (at random time points) in a complete separable metric space, denoted here by $Y$, while its deterministic behaviour between jumps is governed by a finite number of semiflows $S_i$, $i\in I$. These semiflows are randomly switched with time from jump to jump (like e.g. in \cite{b:benaim1}). The state directly after the jump, called a \emph{post-jump location}, is determined by a transformation of the current position, selected randomly from an uncoutable set. Random dynamical systems with the above-described or a similar jump mechanism offer a description of a large class of phenomena arising in variuous domains of natural science, in molecular biology especially (e.g. models for gene expression  \cite{b:dawid, b:lipniacki, b:mackey_tyran} or an autoregulated gene in a bacterium \cite{b:HHS}), but also in population biology (e.g. \cite{b:othmer}) and communication networks (e.g. \cite{b:dumas}).

Most available results on such dynamical systems seem to concentrate on the situation where the jumps occur according to a Poisson process with constant intensity. In this paper we consider the case where the jump intensity depends on the trajectory of the process (as in \cite{b:asia}). To be more precise, letting $\{\tau_n\}_{n\in\n_0}$ denote the sequence of jump times, the conditional probability that the next jump, say $\tau_{n+1}$, will occur before time $t$ has the form
$$\pr(\Delta\tau_{n+1}\leq t\,|\,Y(\tau_n)=y,\,\xi(\tau_n)=i\,)=1-\exp\left(-\int_0^t \lambda(S_i(s,y))\,ds \right),$$
where $\{\xi(t)\}_{t\geq 0}$ is a stochastic process with values in $I$ which indicates the semiflow that currently determines the evolution of the system.

We shall focus on the limit behaviour of the Markov chain $\{(Y_n,\xi_n)\}_{n\in\n_0}$ given by the post-jump locations of the PDMP $\{(Y(t),\xi(t))\}_{t\geq 0}$, that is, defined by \hbox{$(Y_n,\xi_n)=(Y(\tau_n),\xi(\tau_n))$} for $n\in\n_0$. Such a discrete-time dynamical system (in the context presented here) includes as a special case, for instance, a simple cell cycle model examined by Lasota and Mackey \cite{b:cells}, and, furthermore, may prove useful for improvement of the model in \cite{b:HHS}.

Our first goal is to establish a criterion for \emph{exponential ergodicity} of the transition operator associated with the chain $\{(Y_n,\xi_n)\}_{n\in\n_0}$, analogously as in \cite[Theorem 4.1]{b:dawid}. More precisely, letting $(\cdot)P$ stand for the Markov operator acting on Borel measures in such a way that $\mu_{n+1}=\mu_n P$ for $n\in\n_0$, where $\mu_n$ is the distibution of $(Y_n,\xi_n)$, we provide sufficient conditions under which there exists exactly one probability measure $\mu_*$ that is invariant for $P$, i.e. $\mu_*=\mu_* P$. It turns out that such a measure is \emph{exponentially attracting} in the so-called dual bounded Lipschitz distance, which is induced by the \emph{Fortet-Mourier norm} (\cite{b:dudley,b:lj}), denoted by $\norma{\cdot}_{FM}$. We mean by this that there exists a constant $\beta\in\left[0,1\right)$ such that
$$\norma{\mu P^n -\mu _*}_{FM}\leq C(\mu)\beta^n\;\;\;\mbox{for all}\;\;\;n\in\n\;\;\;\mbox{and}\;\;\;\mu\in\Mp^{\rho_c,1},$$
where $\Mp^{\rho_c,1}$ stands for the set of all Borel probability measures on $Y\times I$ satisfying $\int \rho_c(x,x_*)\mu(dx)<\infty$ for some $x_*$, with $\rho_c$ denoting a suitable metric in $Y\times I$. The idea underlying the proof of this result pertains to asymptotic coupling methods, introduced by Hairer \cite{b:hairer}, which in turn boils down to veryfing the assumptions of \cite[Theorem 2.1]{b:kapica} (cf. \cite[Appendix]{b:dawid})

The second part of the paper discusses the ergodicity of the jump chain associated with the Markov process determined by a Poisson driven stochastic differential equation (PDSDE), which is close in spirit to that developed by Lasota and Traple in \cite{b:las_poiss} (cf. also \cite{b:myjak_pois, b:szar_pois}). PDSDEs have quite important applications in biomathematics (e.g. \cite{b:diek1,b:diek2}), physics and engineering (e.g. \cite{b:snyder}), as well as in financial investment models (see e.g. \cite{b:cont}). The research literature abounds with variations on stochastic equations of a similar type. Here, we investigate the version which takes into account the ideas proposed by Horbacz \cite{b:horbacz_poiss} and Kazak \cite{b:kazak}.
Namely, given a Poisson random counting measure $\Np(dt,d\theta)$ on \hbox{$\left[0,\infty\right)\times \Theta$} and functions $a, \sigma$ and $\lambda$, we consider the initial value problem of the form
$$
dY(t)=a(Y(t),\xi(t))\,dt+\int_{\Theta}\sigma(Y(t),d\theta)\,\Np(\Lambda(dt),\theta),\;\;\;Y(0)=Y_0,
$$
with 
$$\Lambda(t)=\int_0^t \lambda(Y(s))\,ds,\;\;\;\mbox{and} \;\;\;\xi(t)=\xi_n\;\;\;\mbox{for}\;\;\;\Np(\left[0,\Lambda(t)\right],\Theta)=n,\;n\in\n_0,$$ 
where $\{Y(t)\}$ is an unknown process with values in a closed subset of a separable Hilbert space, and $\{\xi_n\}_{n\in\n_0}$ is a sequence of random variables with values in a finite set and distributions conditional on the realisation of $\{Y(t)\}_{t\geq 0}$. Assuming that
$$\Np(\left[0,t\right],A)=\operatorname{card}\{n\in\n:\,\btau_n\leq t,\,\eta_n\in A\},$$
one can define a sequence of $\left[0,\infty\right)$-valued random variables $\{\tau_n\}_{n\in\n_0}$ such that $\Lambda(\tau_n)=\overline{\tau}_n$ for $n\in \n_0$.  We shall give a set of quite easily verifiable restrictions on the functions $a, \sigma$ and $\lambda$ ensuring that the solution process $\{Y(t)\}_{t\geq 0}$ determines the PDMP $\{(Y(t),\xi(t))\}_{t\geq 0}$ which we have described so far (with jump times $\tau_n$ conditionally distributed on the solution) , and that the Markov chain $\{(Y(\tau_n),\xi(\tau_n))\}_{t\geq 0}$ is exponentially ergodic. This result generalises \hbox{\cite[Theorem 4.10]{b:kazak}.}

The outline of the paper is as follows. Section \ref{sec:1} contains notation and basic definitions related to the theory of Markov operators. In Section \ref{sec:2} we quote the result of Kapica and \'Sl\k{e}czka \cite{b:kapica}, which we refer to in the next part of the paper. Our main result, providing sufficient conditions for the exponential ergodicty of the dynamical system under consideration, is established in Section \ref{sec:3}. Finally, in Section \ref{sec:4} we give a criterion for the exponential ergodicity of the jump chain related to the aforementioned PDSDE.

\section{Preliminaries} \label{sec:1}
Consider a metric space $(E,\rho )$ endowed with the $\sigma$-field $\B(E)$ of all its Borel subsets. By $B(x,r)$ we denote the open ball in $E$ centered at $x$ and radius $r$. For every set $A\subset E$ we write $\mathbbm{1}_A$  for the indicator function of $A$, i.e. $\mathbbm{1}_A(x)=1$ for $x\in A$ and $\mathbbm{1}_A(x)=0$ otherwise.

Let $\Ms(E)$ stand for the space of all finite, countably additive functions (signed Borel measures) on $\B(E)$. By $\M(E)$ and $\Mp(E)$ we denote the subsets of $\Ms(E)$ consisting of all non-negative measures and all probability measures, respectively. Furthermore, we write $\Mp^{\rho,1}(E)$ for the set of all \hbox{$\mu\in\Mp(E)$} satisfying $\int_E \rho(x,x_*)\,\mu(dx)<\infty$ for some $x_*\in X$.

Moreover, by $B(E)$ we denote the space of all bounded, Borel, real valued functions on $E$, equipped with the supremum norm $\norma{\cdot}_{\infty}$, and we use $C(E)$ to denote its subspace consisting of all continuous functions. For $f \in B(E)$ and $\mu  \in \Ms(E)$ we write 
$$\<f,\mu \>=\int_E f(x)\mu (dx).$$

A continuous function $V:E\rightarrow [0,\infty)$ is called \emph{Lyapunov function} if it is bounded on bounded sets and for some $x_0\in E$,
$$\lim_{\rho(x,x_0)\rightarrow \infty}V(x)=\infty.$$

We say that a sequence $\{\mu_n\}_{n \geq 1}\subset \M(E)$ \emph{converges weakly} to a measure $\mu  \in \M(E)$ whenever
\[\lim_{n\rightarrow \infty}\<f,\mu _n\>=\<f,\mu \> \;\;\;\mbox{ for all}\;\;\; f \in C(E).\]

The set $\Ms(E)$ will be endowed with the so-called \emph{Fortet-Mourier norm} \cite{b:las_frac,b:lj} (equivalent to the \emph{Dudley norm} \cite{b:dudley}), given by 
\[\vectornorm{\mu}_{FM} =\sup\{|\<f,\mu \>|:f \in \mathcal{F}_{FM}(E) \}\;\;\; \textrm{for}\;\;\; \mu  \in \Ms(E),  \] 
where $\mathcal{F}_{FM}(E)$ stands for the set of all  $f \in C(E)$ such that $\norma{f}_{\infty}\leq 1$ and \hbox{$|f(x)-f(y)| \leq \rho (x,y)$} for $x,y \in E.$
It is well-known (cf. \cite{b:dudley}) that, whenever $E$ is a Polish space, the metric $(\mu,\nu)\mapsto\norma{\mu-\nu}_{FM}$ induces the topology of weak convergence of measures in $\Mp(E)$.

A function $P:E\times\mathcal{B}(E)\rightarrow \left[0,1\right]$ is called a \emph{(sub)stochastic kernel} if for each \hbox{$A\in\mathcal{B}(E)$}, $x\mapsto P(x,A)$ is a measurable map on $E$, and for each $x\in E$, $A\mapsto P(x,A)$ is a (sub)probability Borel measure on $\B(E)$. We also sometimes call a stochastic kernel a \emph{transition probability kernel} or, simply, \emph{transition law}.

For an arbitrary given (sub)stochastic kernel $P$  we consider two operators defined by
\begin{equation} \label{regp} \mu P(A)=\int_{E} P(x,A)\,\mu(dx)\;\;\;\mbox{for}\;\;\; \mu\in\M(E),\;A\in \B(E),\end{equation}
and
\begin{equation}\label{regd} Pf(x)=\int_{E} f(y)\,P(x,dy)\;\;\;\mbox{for}\;\;\; x\in E,\;f\in B(E). \end{equation}

If the kernel $P$ is stochastic then the map $(\cdot) P:\M(E) \to \M(E)$ given by \eqref{regp} is called a \emph{regular Markov operator}. It is easy to check that
\begin{equation} \label{duality} \<f,\mu P\>=\<Pf,\mu\>\;\;\;\mbox{for}\;\;\; f\in B(E),\;\mu\in\M(E),\end{equation}
and, therefore, $P(\cdot):B(E)\to B(E)$ defined by \eqref{regd} is said to be the \emph{dual operator} of $(\cdot)P$. Let us note that the dual operator can be extended in the usual way to a linear operator defined on the space of all bounded below Borel functions $\overline{B}(E)$ so that \eqref{duality} holds for all $f\in \overline{B}(E)$. 

A regular Markov operator $P$ is called \emph{Feller} if its dual operator preserves the continuity, i.e. $Pf\in C(E)$ for every $f\in C(E)$. 

A measure $\mu_*\in\M(E)$ is called \emph{invariant} with respect to $P$ if $\mu_* P = \mu_*$. We shall say that such a measure is \emph{attracting} (in $\Mp^{\rho,1}(E)$)  whenever $\mu_*\in \Mp^{\rho,1}(E)$ and for each $\mu \in \Mp^{\rho,1}(E)$,
$$\lim_{n\rightarrow \infty}\norma{\mu P^n -\mu _*}_{FM} =0.$$
If the rate of this convergence is exponential, that is, there exists $\beta\in \left[0,1\right)$ such that, for every $\mu \in \Mp^{\rho,1}(E)$ and some constant $C(\mu)\in\mathbb{R}$, we have
$$\norma{\mu P^n -\mu _*}_{FM}\leq C(\mu)\beta^n\;\;\;\mbox{for all}\;\;\;n\in\n,$$
then $\mu_*$ is called \emph{exponentially attracting}. If such an exponentially attracting invariant probability measure exists then the operator $P$ is said to be \emph{exponentially ergodic}.

It is well known that for every stochastic kernel $P$ and an arbitrary measure \hbox{$\mu_0\in\Mp(E)$} we can always define a discrete-time homogeneus Markov chain $\{\Phi_n\}_{n\in\n_0}$ for which $\mu_0$ will be the distribution of $\Phi_0$, and $P$ will serve as a description of the one-step transition laws, that is
$$\pr(\Phi_0\in A)=\mu_0(A)\;\;\;\mbox{for}\;\;\;A\in\mathcal{B}(E),$$
\begin{equation} \label{trans_n} P(x,A)=\pr(\Phi_{n+1}\in A|\Phi_n=x)\;\;\;\mbox{for}\;\;\;x\in E,\;A\in\mathcal{B}(E),\;n\in\n_0.\end{equation}
Then the Markov operator  $(\cdot)P$  corresponding to the kernel \eqref{trans_n} decribes the evolution of the distributions $\mu_n:=\pr(\Phi\in \cdot)$, that is
$$\mu_{n+1}=\mu_n P\;\;\;\mbox{for}\;\;\;n\in\n_0. $$
In our further considerations we will use the symbol $\pr_x$ for the probability measure $\pr(\cdot\,|\,\Phi_0=x)$ and $\ew_x$ for the expectation with respect to $\pr_x$.

Assuming that a stochastic kernel \hbox{$P:E\times \mathcal{B}(E)\to\left[0,1\right]$} is given, we will say that a time-homogeneus Markov chain evolving on the space $E^2$ (endowed with the product topology) is  a \emph{{Markovian} coupling} of $P$  whenever its transition law \hbox{$B:E^2\times \mathcal{B}(E^2)\to\left[0,1\right]$} satisfies
$$
B(x,y,A\times E)= P(x,A)\;\;\;\mbox{and}\;\;\; B(x,y,E\times A)= P(y,A)
$$
for all $x,y\in E$ and $A\in\B(E)$. Note that, if $Q: E^2\times\mathcal{B}(E^2)\to\left[0,1\right]$ is a substochastic kernel satisfying
\begin{equation} 
\label{qeq}
Q(x,y,A\times E)\leq P(x,A)\;\;\;\mbox{and}\;\;\; Q(x,y,E\times A)\leq P(y,A),
\end{equation}
for all $x,y\in E$ and $A\in\B(E)$, then we can always construct a {Markovian} coupling of $P$ whose transition function $B$ satisfies $Q\leq B$. Indeed, it suffices to define the family \hbox{$\{R(x,y,\cdot):\;x,y\in E\}$} of measures on $\mathcal{B}(E^2)$, which on rectangles $A\times B\in\mathcal{B}(E^2)$ are given by
\begin{align*}
R(x,y,A\times B)&=
\frac{1}{1-Q(x,y,E^2)}\left[P(x,A)-Q(x,y,A\times E))\right.\\
&\quad\times\left.(P(y,B)-Q(x,y,E\times B)\right]
\end{align*}
if $Q(x,y,E^2)<1$, and $R(x,y, A\times B)=0$ otherwise. It is then easily seen that \hbox{$B:=Q+R$} is a stochastic kernel satisfying $Q\leq B$, and that the Markov chain with transition function $B$ is a coupling of $P$. 

\section{A general criterion on the exponential ergodicity for Markov--Feller operators}\label{section:auxiliary} \label{sec:2}
In this section we quote \cite[Theorem 2.1]{b:kapica} of Kapica and \'Sl\k{e}czka (cf. also \cite[Appendix]{b:dawid}), which provides sufficient conditions for a Markov operator to possess an exponentially attracting invariant probability measure. This is a crucial tool in the proof of our main result, given in the next section.

\begin{thm}\label{ks-stab}
Suppose that \hbox{$P:\M(E)\to \M(E)$} is a Markov operator, which enjoys the Feller property, and that there exists a substochastic kernel $Q$ on $E^2\times\mathcal{B}(E^2)$ satisfying \eqref{qeq}. Furthermore, assume that the following conditions hold:
\begin{itemize}
\phantomsection
\item[(B1)]\label{cnd:B1}There exist a Lyapunov function $V:E\rightarrow\left[0,\infty\right)$ and constants $a\in(0,1)$ and $b>0$ satisfying
$$PV(x)\leq aV(x)+b\;\;\;\mbox{for}\;\;\;x\in E.$$

\phantomsection
\item[(B2)]\label{cnd:B2}For some $F\in\mathcal{B}(E^2)$ and some $R>0$ the following conditions are satisfied:
\begin{itemize}
\item[\sbullet] $\supp Q(x,y,\cdot)\subset F$ for $(x,y)\in F$;
\item[\sbullet] There exists a {Markovian} coupling $\{(\Phi^1_n,\Phi^2_n)\}_{n\in\n_0}$ of $P$ with transition function $B$, {\hbox{satisfying}} $Q\leq B$, such that for
\begin{equation} \label{zb_k} K:=\{(x,y)\in F:\, V(x)+V(y)<R\} \end{equation}
and $\sigma:=\inf\{n\in\n:\,(\Phi^1_n,\Phi^2_n)\in K\}$ we can choose constants \hbox{$\zeta\in (0,1)$} and $\bar{C}>0$ so that
{\begin{equation} \label{fastint} \ew_{(x,y)}(\zeta^{-\sigma})\leq \bar{C}\quad\mbox{whenever}\;\;V(x)+V(y)<\frac{4b}{1-a}.\end{equation}}
\end{itemize}

\phantomsection
\item[(B3)]\label{cnd:B3}There exists a constant $q\in (0,1)$ such that
$$\int_{E^2}\rho(u,v)\, Q(x,y,du,dv)\leq q\rho(x,y)\;\;\;\mbox{for}\;\;\;(x,y)\in F.$$

\phantomsection
\item[(B4)]\label{cnd:B4}Letting $U(r)=\{(x,y):\, \rho(x,y)\leq r\}$ for $r>0$, we have
$$\inf_{(x,y)\in F}Q(x,y,U(q\rho(x,y)))>0.$$

\phantomsection
\item[(B5)]\label{cnd:B5}There exist constants $l>0$ and $\nu\in\left(0,1\right]$ such that $$Q(x,y,E^2)\geq 1-l\rho(x,y)^{\nu} \;\;\; \mbox{for} \;\;\; (x,y)\in F.$$
\end{itemize}
Then the operator $P$ possesses a unique invariant measure $\mu^{*}\in \Mp(E)$ and \hbox{$\<V,\mu^*\><\infty$}. Moreover, there exist constants \hbox{$\beta\in \left[0,1\right)$} and $C\in\mathbb{R}$ such that
\begin{equation} \label{rate} \norma{\mu P^n-\mu^{*}}_{FM}\leq C\beta^n(\<V,\,\mu+\mu^*\>+1) \end{equation}
for all $n\in\n$ and every {$\mu\in\Mp(E)$} satisfying $\<V,\mu\><\infty$. 
\end{thm}

The key idea underlying the above result is the existence of a Markovian coupling whose transition function can be decomposed into two substochastic kernels, of which one, denoted by $Q$, enjoys the contractivity property \hyperref[cnd:B3]{(B3)} and plays a dominant role in the evolution of the coupled Markov chain. By the dominance of $Q$ we mean that there exists a finite (random) time, say $\tau$, from which onwards the next step of the coupled chain is drawn only according to $Q$, and the probability that $\tau$ occurs soon after the chain starts is small. More precisely, conditions \hyperref[cnd:B1]{(B1)}-\hyperref[cnd:B5]{(B5)} guarantee that for some $\gamma\in (0,1)$ and $c\in\mathbb{R}$ we have $\pr_{(x,y)}(\tau>n)\leq c\gamma^n(V(x)+V(y)+1)$ for all $n\in\n$ and $(x,y)\in E^2$ (which follows from  \cite[Lemmas 2.1 and 2.2]{b:kapica}). Such a dominant contractive part $Q$ makes the copies of the Markov chain (governed by $P$) possible to couple at infinity in an exponential rate. The proof of this theorem is based on certain asymptotic coupling techniques, introduced in \cite{b:hairer} by Hairer.

\section{Exponential ergodicity for some random dynamical system}\label{sec:3}
This section contains the main result of the paper. As mentioned in Introduction, we shall consider a piecewise-deterministic Markov process with a state-dependent jump intensity (as in \cite{b:asia}; cf. \cite{b:horbacz_diss, b:dawid}), whose deterministic evolution is governed by a finite collection of semiflows. The subject of our investigation is the exponential ergodicity of the Markov chain given by the post-jump locations of such a process.
\subsection{Model description and assumptions}
Let $(Y,\rho)$ be a Polish space, i.e. a complete separable metric space, and set $\R_+=[0,\infty)$. Further, let  $\Theta $ be a compact interval, and put $I=\{1,...,N\}$, where $N$ is an arbitrary (and fixed) positive integer. Now define $X:=Y\times I$. In what follows, the space $X$ is endowed with the metric $\rho_c$ given by
\begin{equation}\label{def:roc} \rho_c((y_1,i),(y_2,j))=\rho(y_1,y_2)+c\delta (i,j)\;\;\;(y_1,i),(y_2,j)\in X,\end{equation}
where
\begin{equation}
\label{def:dirack}
\delta(i,j)= \begin{cases}
1 \textrm{\;\;for\;\;} i\neq j, \\
0 \textrm{\;\;for\;\;} i= j,
\end{cases}
\end{equation}
and $c$ is a positive constant specified at the end of this section.

We shall investigate a discrete-time dynamical system arising from a stochastic process $\{(Y(t),\xi(t))\}_{t\geq 0}$, which evolves through random jumps in the space $X$ as described below.

Assume that we are given a finite collection of semiflows ${S_i:\R_+\times Y}\rightarrow Y, $ $i\in I$, which are  continuous with respect to each variable. The \emph{semiflow property} means, as usual, that, for every $i\in I$ and each $y\in Y$,
$$S_i(0,y)=y\;\;\;\mbox{and}\;\;\;S_i(s+t,y)=S_i(s, S_i(t,y))\;\;\; \mbox{for} \;\;\;s,t \in \R_+.$$
Between jumps the process $\{Y(t)\}_{t\geq 0}$ is deterministic and evolves according to one of the semiflows, whose index is determined by $\{\xi(t)\}_{\geq 0}$.

The semiflows $S_i$ are being switched (directly after the jumps) according to a matrix of continuous functions
$\pi_{ij}:Y\rightarrow [0,1]$, $i,j\in I$ (called \emph{probabilities}), satisfying
\begin{equation}
\label{stoch_pi}
\sum_{j\in I} \pi_{ij}(y)=1\;\;\;\mbox{for}\;\;\;i\in I,\, y\in Y.
\end{equation}

At the time of a jump the process $\{Y(t)\}_{t\geq 0}$ moves to the new state determined by a transfomation $q_{\theta}:Y\to Y$, which is randomly drawn from a given family \hbox{$\{q_{\theta}:\,\theta\in\Theta\}$}. It is required that \hbox{$Y\times\Theta\ni(y,\theta) \mapsto q_{\theta}(y)\in Y$} is continuous.
Futhermore, we consider a collection of probability density functions $\Theta \ni \theta \mapsto p_{\theta}(y)$, $y\in Y$, such that $(\theta,y)\mapsto p_{\theta}(y)$ is continuous. These place-dependent densities are related to the likelihood of occurrence of $q_{\theta}$ at the jump times.
Finally, we assume that the intensity of jumps is determined by a Lipschitz continuous function \hbox{$\lambda:Y\rightarrow (0,\infty)$} such that 
\begin{equation}
\label{cnd:a7}
\ulambda=\inf_{y\in Y}\lambda(y)>0\;\;\;\mbox{and}\;\;\;
\olambda=\sup_{y\in Y}\lambda(y)<\infty.
\end{equation}
Let us also define $L:\mathbb{R}_+\times Y\times I\to \mathbb{R}_+$ by
\begin{equation} 
\label{defL}
L(t,(y,i))=\int_0^t\lambda (S_i(s,y))ds.
\end{equation}

In summary, the evolution of $\{(Y(t),\xi(t))\}_{t\geq 0}$ can described roughly as follows. Suppose that the process starts at $(y_0,i_0)\in Y\times I$. We then have $Y(t)=S_{i_0}(t,y_0)$ and \hbox{$\xi(t)=i_0$} until some random time $t_1$ (which depends on $y_0$ and $i_0$). At this moment the process $\{Y(t)\}_{t\geq 0}$ jumps to the new position $y_1:=q_{\theta_1}(Y(t_1-))=q_{\theta_1}(S_{i_0}(t_1, y_0))$, where $\theta_1\in \Theta$ is selected randomly according to the distribution with density \hbox{$\theta \mapsto p_{\theta}(S_{i_0}(t_1, y_0))$}.  Directly after this, we randomly draw \hbox{$i_1\in I$} in such a way that the probability of choosing $i_1=i$ is equal to $\pi_{i_0 i}(y_1)$. Then, from $t=t_1$ onwards, $Y(t)=S_{i_1}(t-t_1,Y(t_1))$ and $\xi(t)=i_1$  until the next jump time $t_2$. The procedure is then repeated with $(y_1,i_1)$ (replacing $(y_0,i_0)$) and is continued inductively. Hence, assuming that $t_0=0$ and $t_n\to\infty$ a.s., we obtain 
$$Y(t)=S_{i_n}(t-t_n,y_n)\;\;\;\mbox{and}\;\;\;\xi(t)=i_n\;\;\;\mbox{for}\;\;\;t\in \left[t_n,t_{n+1}\right),\,n\in\n_0.$$

In this paper, we only focus on the sequence of random variables given by the post-jump locations of such a process, that is $(Y_n,\xi_n):=(Y(\tau_n),\xi(\tau_n))$, $n\in\n_0$, where $\tau_n$ is a random variable describing the jump time $t_n$.

In order to formalise the model, on a suitable probability space, say $(\Omega, \mathcal{F}, \pr)$, we define $\{(Y_n,\xi_n)\}_{n\in\n_0}$ as follows. Let \hbox{$Y_0:\Omega\to Y$} and $\xi_0: \Omega\to I$ be random variables with arbitrary and fixed distributions. Further, define inductively the sequences $\{\tau_n\}_{n\in\n_0}$, $\{\xi_n\}_{n\in \n}$, $\{\eta_n\}_{n\in\n}$ and $\{Y_n\}_{n\in\n}$ of random variables, describing $\{t_n\}_{n\in\\n}$, $\{i_n\}_{n\in\n}$, $\{\theta_n\}_{n\in\n}$ and $\{y_n\}_{n\in\n}$, respectively, so that the following conditions are valid:
\begin{itemize}
\item[\sbullet] $\tau_n :\Omega\to \left[0,\infty\right)$, $n\in\n_0$, with $\tau_0=0$, form a strictly increasing sequence such that $\tau_n\to\infty$ a.e., and the increments \hbox{$\Delta \tau_{n}=\tau_{n}-\tau_{n-1}$} are mutually independent and have the conditional distributions given by
\begin{equation}
\label{tdist}
\pr(\Delta \tau_{n+1}\leq t\,|\,Y_n=y\;\;\mbox{and}\;\;\xi_n=i)=1-e^{-L(t,(y,i))}\;\;\mbox{for}\;\; t\geq 0,
\end{equation}
whenever $y\in Y$ and $i\in I$, where $L$ is given by \eqref{defL}. 

\item[\sbullet] $\xi_n :\Omega\to I$, $n\in \n$, satisfy
$$\pr(\xi_{n}=j\;|\;Y_{n}=y,\, \xi_{n-1}=i)=\pi_{ij}(y)\;\;\;\mbox{for}\;\;\; i, j\in I,\, y\in Y;$$

\item[\sbullet] $\eta_n:\Omega \to \Theta$, $n\in\n$, is specified by
$$ \pr(\eta_{n+1}\in D\;|\;S_{\xi_{n}}(\Delta \tau_{n+1},Y_{n})=y)
=\int_{D} p_{\theta}(y)\,d\theta$$
for all $D\in\mathcal{B}(\Theta)$ and $y\in Y$.
\item[\sbullet]  $Y_n:\Omega\to Y$, $n\in\n$, are determined by $$Y_{n+1}=q_{\eta_{n+1}}(S_{\xi_n}({\Delta \tau_{n+1}},Y_n))\;\;\;\mbox{for}\;\;\;n\in\n_0.$$
\end{itemize}
Moreover, considering
$$U_0=(Y_0,\xi_0),\;\;\;U_n=(Y_0,\tau_1,\ldots,\tau_k,\eta_1,\ldots,\eta_k,\xi_0,\ldots,\xi_k) \;\;\;\mbox{for}\;\;\;k\in\n,$$
we assume that, for every $k\in\n_0$, the random variables $\xi_{k+1}$ and $\eta_{k+1}$ are conditionally independent of $U_k$ given $\{Y_{k+1}=y,\,\xi_k=i\}$ and $\{S_{\xi_{k}}(\Delta \tau_{k+1},Y_{k})=y\}$, respectively. In addition to this, we require that $\xi_{k+1}$, $\eta_{k+1}$ and $\Delta \tau_{k+1}$ are mutually conditionally independent given $U_k$, and that $\Delta \tau_{k+1}$ is independent of $U_k$. 

Let us now define $\Delta_t: X\times\B(X)\to\left[0,1\right]$, $t\geq 0$, by
$$\Delta_t((y,i),A):=\sum_{j\in I}\int_{\Theta} \mathbbm{1}_A(q_{\theta}(S_i(t,y)),j)\pi_{ij}(q_{\theta}(S_i(t,y))p_{\theta}(S_i(t,y))\,d\theta.$$
It is then not hard to check that $\{(Y_n,\xi_n)\}_{n\in\n_0}$ is a time-homogenus Markov chain with phase space $X$ and transition law $P:X\times \B(X)\to\left[0,1\right]$ given by
\begin{equation} \label{def:stP} P((y,i),A)=\int_0^{\infty} \lambda(S_i(t,y))e^{-L(t,(y,i))}\Delta_t ((y,i),A)\,dt,\end{equation}
where $L$ is given by \eqref{defL}. The evolution of the distributions \hbox{$\mu_n:=\pr((Y_n,\xi_n)\in \cdot)$} can be then described by the Markov operator $(\cdot)P:\M(X)\to\M(X)$ corresponding to  \eqref{def:stP}, as defined in \eqref{regp}. Such an operator is sometimes called a \emph{jump operator}, since it describes the distributions of the process  $(Y(t),\xi(t))_{t\geq 0}$ (which is, clearly, also a Markov process) from jump to jump. 
 
In the analysis that follows, we will use the assumptions specified below. 

\begin{enumerate}[label={(A\arabic*)}] %\Alph*

\phantomsection
\label{cnd:a1}
\item There exists $y_*\in Y$ such that
\begin{equation} 
\sup_{y\in Y} \int_0^{\infty}e^{-\ulambda t}\int_{\Theta} \rho(q_{\theta}(S_i(t,y_*)),y_*) p_{\theta}(S_i(t,y))d\theta\,dt<\infty\;\;\mbox{for}\;\;i\in I;
\end{equation}

\phantomsection
\label{cnd:a2}
\item There exist $\alpha \in \R$, $L>0$ and a bounded on bounded set function \hbox{$\L:Y\to \R_+$} such that
\begin{equation}
\rho(S_i(t, y_1), S_j(t, y_2))\leq Le^{\alpha t}\rho(y_1,y_2)+\L(y_2)\delta(i,j)
\end{equation}
for $t\geq 0$, $y_1, y_2  \in Y$, $i,j\in I$, where $\delta(i,j)$ is given by \eqref{def:dirack};

\phantomsection
\label{cnd:a3}
\item There exists a constant $L_q>0$ such that
\begin{equation} 
\int_{\Theta}\rho(q_{\theta}(y_1),q_{\theta}(y_2))\,p_{\theta}(y_1)d\theta\leq L_q\rho(y_1,y_2)\;\;\; \textrm{for}\;\;\; y_1, y_2  \in Y;
\end{equation}

\phantomsection
\label{cnd:a4}
\item There exists $L_{\lambda}>0$ such that
$$|\lambda(y_1)-\lambda(y_2)|\leq L_{\lambda}\rho(y_1,y_2)\;\;\;\mbox{for}\;\;\;y_1,y_2\in Y;$$

\phantomsection
\label{cnd:a5}
\item There exists $L_{\pi}>0$ and $L_p>0$ such that
\begin{equation} 
\begin{split}
&\sum_{j\in I}|\pi_{ij}(y_1)-\pi_{ij}(y_2)|\leq L_{\pi}\rho(y_1,y_2) \;\;\;\textrm{for}\;\;\; y_1, y_2 \in Y,\,i\in I, \\
&\int_{\Theta}|p_{\theta}(y_1)-p_{\theta}(y_2)|\,d\theta\leq L_{p}\rho(y_1,y_2)\;\;\;\mbox{for}\;\;\;y_1,y_2\in Y;\\
\end{split}
\end{equation}

\phantomsection
\label{cnd:a6}
\item There exists $\delta _{\pi}>0$ and $\delta _p>0$ such that
\begin{equation} 
\begin{split}
& \sum_{j\in I}\min\{\pi_{i_1j}(y_1),  \pi_{i_2j}(y_2)  \}\geq \delta_{\pi} \;\;\;\textrm{for}\;\;\; y_1, y_2  \in Y,\; i_1,i_2\in I,\\
&\int_{\Theta(y_1,y_2)}\min\{p_{\theta}(y_1),  p_{\theta}(y_2)  \}d\theta\geq\delta _p \;\;\;\textrm{for}\;\;\; y_1, y_2  \in Y,
\end{split}
\end{equation}
where $\theta(y_1,y_2)=\{\theta\in\Theta:\, \rho(q_{\theta}(y_1),q_{\theta}(y_2))\leq \rho(y_1,y_2)\}$.
\end{enumerate}

As we have mentioned earlier, we also require that $c$ appearing in \eqref{def:roc} is sufficiently large. The choice of $c$ depends on constants appearing in conditions  \hyperref[cnd:a1]{(A1)}-\hyperref[cnd:a4]{(A4)}. We assume, namely, that
\begin{equation}\label{nc} c\geq  
\max\left\{\frac{\olambda}{\ulambda-\alpha}, e^{\sup T},\,\frac{1}{\ulambda}
\right\}\frac{(\ulambda-\alpha)M_{\mathcal{L}}}{\ulambda L}+\frac{2(\ulambda-\alpha)}{\ulambda L},  
\end{equation}
where $T\subset \left[0,\infty\right)$ is a fixed bounded set with positive measure such that
\begin{equation}\label{e:remT} e^{\alpha t}\leq \frac{\olambda}{\ulambda-\alpha}\;\;\;\mbox{for all}\;\;\;t \in T, \end{equation}
and 
\begin{equation}\label{ml} M_{\mathcal{L}}:=\sup\{\mathcal{L}(y):\rho(y,y_*)<R\}\;\;\mbox{with}\;\;\;R:=4b/(1-a),\end{equation}
where
\begin{equation}\label{def:ab}
\begin{split}
a:&=(\olambda LL_q)(\ulambda - \alpha)^{-1},\\
b:&= \olambda\max_{i\in I} \sup_{y\in Y}\int_0^{\infty} e^{-\ulambda t} \int_{\Theta}\rho(q_{\theta}(S_i(t,y^*)),y_*)p_{\theta}(S_i(t,y))\,d\theta\,dt.
\end{split}
\end{equation}

\subsection{Exponential ergodicity of the jump opertor}
In order to establish the exponential ergodicity of the Markov operator $P$ corresponding to $\{(Y_n,\xi_n)\}_{n\in\n_0}$, we want to use Theorem \ref{ks-stab} (a similar approach is also taken e.g. in \cite{b:dawid}). For this reason, we introduce the following piece of notation, which will be useful for the rest of the section. For any $x_1=(y_1,i_1)$, \hbox{$x_2=(y_2,i_2)\in X$}, $t\geq 0$, $\theta\in\Theta$ \hbox{we put}
\begin{equation}
\begin{split}
&\textbf{p}_{\theta}(x_1,x_2,t)=p_{\theta}(S_{i_1}(t,y_1))\wedge p_{\theta}(S_{i_2}(t,y_2)), \\
&\boldsymbol{\pi}_j(x_1,x_2,t,\theta )=\pi _{i_1j}(q_{\theta}(S_{i_1}(t,y_1)))\wedge \pi _{i_2j}(q_{\theta}(S_{i_2}(t,y_2))),\\
&\textbf{q}_j(x_1,x_2,t,\theta)= (( q _{\theta}(S_{i_1}(t,y_1)),j),(q_{\theta}(S_{i_2}(t,y_2)),j)),\\
&\boldsymbol{\lambda}(x_1,x_2,t)=\lambda(S_{i_1}(t,y_1))\wedge\lambda(S_{i_2}(t,y_2)),\\
&\boldsymbol{L}(x_1,x_2,t)=e^{-L(t,y_1,i_1)} \wedge e^{-L(t,y_2,i_2)},
\end{split}
\end{equation}
where $\wedge$ denotes minimum in $\mathbb{R}_+$.

Let us now consider the space $X^2$ with the following metric:
$$\overline{\rho}_c((x_1,x_2),(z_1,z_2))=\rho_c(x_1,z_1)+\rho_c(x_2,z_2)\;\;\; \textrm{for} \;\;\;(x_1,x_2),(z_1,z_2)\in X^2.$$
Further, for each $t\geq 0$, let $\Gamma_t:X^2\times\B(X^2)\to\left[0,1\right]$ be given by
\begin{equation}
\Gamma_t(x_1,x_2,A)=\sum_{j\in I} \int_{\Theta}\mathbbm{1}_A(\textbf{q}_j(x_1,x_2,t,\theta))\boldsymbol{\pi}_j(x_1,x_2,t,\theta )\textbf{p}_{\theta}(x_1,x_2,t)d\theta,\\
\end{equation}
and define $Q:X^2\times\B(X^2)\to\left[0,1\right]$ by
\begin{equation}
\label{def:Q}
Q(x_1,x_2,A)=\int_0^{\infty}\boldsymbol{\lambda}(x_1,x_2,t)\boldsymbol{L}(x_1,x_2,t)\Gamma_t(x_1,x_2,A)dt.
\end{equation}
It is then easily seen that $Q$ is a substochastic kernel satisfying \eqref{qeq}.

We are now in a position to establish our main result.
\begin{thm}\label{thm_main}
Suppose that conditions \hyperref[cnd:a1]{(A1)}-\hyperref[cnd:a6]{(A6)} hold with
\begin{equation}
\label{cnd:LL}
LL_q\olambda+\alpha< \ulambda.
\end{equation} 
Then the Markov operator $P$ corresponding to \eqref{def:stP} has a unique invariant distribution $\mu_*$, which is exponentially attracting. More precisely, \hbox{$\mu_*\in\Mp^{\rho_c,1}(X)$} and there exist $x_*\in X$, $C \in \R$ and $\beta\in[0,1)$ such that
\begin{equation}
\norma{\mu P^n-\mu_*}_{FM}\leq C \beta^n(\int_X\rho_c(x_*,x)(\mu+\mu_*)(dx)+1)
\end{equation}
for any $\mu\in\Mp^{\rho_c,1}(X)$ and any $n\in\n$.
\end{thm}
\begin{proof}
It suffices to show that all hypotheses of Theorem \ref{ks-stab} hold for $P$ and $Q$ defined by \eqref{def:stP} and \eqref{def:Q}, respectively. First of all, observe that $P$ is Feller, which follows immediately from the continuity of functions  $\pi_{i,j}$, $y\mapsto S_i(t,y)$, $y\mapsto p(y,\theta)$ and $q_{\theta}$ for all $i,j\in I$, $t\geq 0$ and $\theta\in\Theta$. Moreover, notice that \eqref{cnd:LL} in particular implies that $\alpha<\ulambda$. This will guarantee the finiteness of the integrals containing $e^{(\alpha-\ulambda)t}$, which occur in the analysis below. Our further reasoning falls naturally into five parts.

\textbf{Step 1.} Let $V:X\rightarrow [0,\infty)$ be defined by
\begin{equation}
\label{def:v}
V(y,i)=\rho(y,y_*)\;\;\; \textrm{for}\;\;\; (y,i)\in X,
\end{equation}
where $y_*$ is specified in \hyperref[cnd:a1]{(A1)}. Clearly, $V$ is a Lyapunov function, and taking $x_*:=(y_*,i_*)$ (with an arbitrary and fixed $i_*\in I$) we get $V(x)\leq \rho_c(x^*,x)$ for all $x\in X$. Now fix $(y,i)\in X$, and let $a$ and $b$ denote the constants given by \eqref{def:ab}. Further, define
$$B_j(t)=\int_{\Theta}\rho(q_{\theta}(S_j(t,y_*)),y_*)\,p_{\theta}(S_j(t,y))\,d\theta\;\;\; \mbox{for}\;\;\;t\geq 0,\,j\in I.$$
From \eqref{cnd:LL} and \hyperref[cnd:a1]{(A1)} it follows that $a\in (0,1)$ and $b<\infty$, respectively. In particular we then see that $B_j(t)<\infty$ for almost all $t\geq 0$. Using conditions \hyperref[cnd:a3]{(A3)} and \hyperref[cnd:a2]{(A2)}, sequentially, we conclude that
\begin{align}
\Delta_t V(y,i)&=\int_{\Theta} \rho(q_{\theta}(S_i(t,y)),y_*))
\Bigl[\,\sum_{j\in I}\pi_{ij}(q_{\theta}(S_i(t,y)))\,\Bigr]p_{\theta}(S_i(t,y))\, d\theta\\
&\leq \int_{\Theta} \rho(q_{\theta}(S_i(t,y)),q_{\theta}(S_i(t,y_*)))\,p_{\theta}(S_i(t,y))\, d\theta\\
&\quad + \int_{\Theta} \rho(q_{\theta}(S_i(t,y_*)),y_*)\,p_{\theta}(S_i(t,y))\, d\theta\\
&\leq L_q\rho(S_i(t,y),S_i(t,y_*))+B_i(t)\leq LL_q e^{\alpha t}\rho(y,y_*)+B_i(t).
\end{align}
Consequently, having in mind \eqref{cnd:a7}, we obtain
\begin{align}
\label{eq:av}
\begin{split}
P V(y,i)&=\int_0^{\infty} \lambda(S_i(t,y))e^{-L(t,y,i)}\, \Delta_t V(y,i)\leq \int_0^{\infty} \olambda e^{-\ulambda t}\, \Delta_t V(y,i)
 \, dt\\
& \leq \olambda L L_q \int_0^{\infty} e^{(\alpha-\ulambda)t}\,dt\,\rho(y,y_*)+\int_0^{\infty} \olambda e^{-\ulambda t} B_i(t)\,dt\\
&\leq \frac{\olambda L L_q }{\ulambda-\alpha}\rho(y,y_*)+b=a V(y,i)+b,
\end{split}
\end{align}
which establishes \hyperref[cnd:B1]{(B1)}.

\textbf{Step 2.} Let us define $R:=4b/(1-a)$ and the following subsets of $X^2$:
\begin{gather}
F_1:=\{((y_1,i_1),(y_2,i_2)): i_1=i_2\},\\ 
F_2:=\{((y_1,i_1),(y_2,i_2)): V(y_1,i_1)+V(y_2,i_2)<R\}.
\end{gather}
We will prove that \hyperref[cnd:B2]{(B2)} is satisfied with the above defined $R$ and \hbox{$F:=F_1 \cup F_2.$}

We first need to show that $\supp Q(x_1,x_2,\cdot)\subset F$ for  every \hbox{$(x_1,x_2)\in X^2.$}
For this purpose, let 
$$(x_1,x_2):=((y_1,i_1),(y_2,i_2))\in X^2\;\;\;\mbox{and}\;\;\;(z_1,z_2):=((u_1,k_1),(u_2,k_2))\in X^2\backslash F.$$ Since $k_1 \not= k_2,$ it follows that
\begin{equation}
\overline{\rho}_c(\textbf{q}_j(x_1,x_2,t,\theta),(z_1,z_2))\geq c(\delta(j,k_1) +\delta(j,k_2))_\geq c \;\;\;\textrm{for} \;\;\;j \in I,\, t\geq 0.
\end{equation}
Hence, letting $\gamma\in(0,c)$, we obtain $\textbf{q}_j(x_1,x_2,t,\theta)\not\in B((z_1,z_2),\gamma)$ for all $j\in I$ and $t\geq 0$.
This implies that $\Gamma_t(x_1,x_2,B((z_1,z_2),\gamma))=0$ for all $t\geq 0$, and in turn yields  that $Q(x_1,x_2,B((z_1,z_2),\gamma))=0$. Consequently, we see that $(z_1,z_2)\in X^2\backslash \supp Q(x_1,x_2,\cdot)$,  as claimed. 

Let $\{(X^1_n,X^2_n)\}_{n\in\n_0}$ be an arbitrary Markovian coupling of $P$ with transition function $B$ such that $Q\leq B$, and define
\begin{equation}
\sigma = \inf\{n \in \n_0 : (X^1_n,X^2_n) \in F,\;\; V(X^1_n) + V(X^2_n) < R\}.
\end{equation}
Consider the Lyapunov function $\overline{V}:X^2\rightarrow [0,\infty)$ defined by
\begin{equation}
\overline{V}(x_1,x_2)=V(x_1)+V(x_2) \;\;\;\textrm{for}\;\;\; (x_1,x_2)\in X^2.
\end{equation}
Since $\overline{V}(x_1,x_2)<R$ for $(x_1,x_2)\in F$, we have
\begin{equation}
\sigma= \inf\{n \in \n_0: \overline{V}(X^1_n,X^2_n)< R\}.
\end{equation}
Moreover, in view of \eqref{eq:av}, it is easy to see that
\begin{equation}
B\overline{V} (x_1; x_2)\leq \overline{V} (x_1; x_2) + 2b\;\;\; \textrm{for}\;\;\;(x_1, x_2) \in X^2.
\end{equation}
Hence the proof of \hyperref[cnd:B2]{(B2)} is completed by using [21, Lemma 2.2].

\textbf{Step 3.} Set $q:=a=(\olambda LL_q)(\ulambda - \alpha)^{-1}$, and let \hbox{$(x_1,x_2):=((y_1,i_1),(y_2,i_2))\in F$}. Applying \hyperref[cnd:a3]{(A3)} we see that
\begin{align}
\label{st3:a}
\begin{split}
\Gamma_t \rho_c(x_1,x_2)&=\int_{\Theta}\rho _c(\textbf{q}_j(x_1,x_2,t,\theta))\textbf{p}_{\theta}(x_1,x_2,t)\Big(\sum_{j\in I}\boldsymbol{\pi}_j(x_1,x_2,t,\theta)\Big)\,d\theta\\
&\leq \int_{\Theta} \rho(q_{\theta}(S_{i_1}(t,y_1)),q_{\theta}(S_{i_2}(t,y_2)))p_{\theta}(S_{i_1}(t,y_1))\,d\theta\\
&\leq L_q \rho(S_{i_1}(t,y_1),S_{i_2}(t,y_2))\;\;\;\mbox{for}\;\;\;t\geq 0.\\ 
\end{split}
\end{align}
If $i_1\neq i_2$ then $(x_1,x_2)\in F\backslash F_1=F_2$, and thus, due to the definition of $M_{\mathcal{L}}$ given in \eqref{ml}, we have $\mathcal{L}(y_2)\leq M_{\mathcal{L}}$. Hence, it follows from \hyperref[cnd:a2]{(A2)} that
\begin{equation}
\label{st3:m}
\rho(S_{i_1}(t,y_1),S_{i_2}(t,y_2))\leq L e^{\alpha t} + tM_{\mathcal{L}}\delta(i_1,i_2)\;\;\;\mbox{for}\;\;\;t\geq 0.
\end{equation}
Combining this with \eqref{st3:a} gives
$$\Gamma_t q_c(x_1,x_2) \leq L_q(Le^{\alpha t}\rho(y_1,y_2)+t M_{\mathcal{L}}\delta(i_1,i_2))\;\;\;\mbox{for}\;\;\;t\geq 0.$$
Finally, applying \eqref{cnd:a7} and \eqref{nc}, we obtain
\begin{equation}
\begin{split}
Q q_c&(x_1,x_2)\leq\int_0^{\infty}\lambda(S_i(t,y))e^{-L(t,y,i)}\,\Gamma_t q_c(x_1,x_2)dt\\
&\leq \int_0^{\infty}\olambda e^{-\ulambda t}\,\Gamma_t q_c(x_1,x_2)dt\\
&\leq \olambda L L_q \left(\int_0^{\infty} e^{(\alpha-\ulambda)t}\,dt\,\rho(y_1,y_2)+\frac{M_{\mathcal{L}}}{L}\int_0^{\infty} t e^{\ulambda t}\,dt\,\delta(i_1,i_2)\right)\\
&\leq \frac{\olambda LL_q}{\ulambda-\alpha}(\rho(y_1,y_2)+\frac{(\ulambda-\alpha)M_{\mathcal{L}}}{\ulambda^2L}\delta(i_1,i_2))\leq q\rho _c(x_1,x_2).
\end{split}
\end{equation}

\textbf{Step 4.} Let $T\subset\left[0,\infty\right)$ be the bounded set with positive measure for which \eqref{e:remT} holds, and put $\delta:=\delta _{\pi}\delta _p\int_T\ulambda e^{-\olambda t}dt$. Using \eqref{e:remT} we obtain  
\begin{equation}
\label{st4:ea}
LL_q e^{\alpha t}\leq q \;\;\;\textrm{for}\;\;\; t\in T\;\;\mbox{(where } q=a\mbox{)}.
\end{equation}
Let $(x_1,x_2):=((y_1,i_1),(y_2,i_2))\in F$, and define
$$
U:=\{(u_1,u_2)\in X^2: \rho_c(u_1,u_2)\leq q\rho_c(x_1,x_2)\}.
$$
We will show that $Q(x_1,x_2,U)\geq \delta.$ For this purpose, let us consider the following sets:
$$R_1(t)=\{\theta \in \Theta: \rho (q_{\theta}(S_{i_1}(t,y_1)),q_{\theta}(S_{i_2}(t,y_2)))\leq L_q\rho (S_{i_1}(t,y_1),S_{i_2}(t,y_2))\},$$
$$R_2(t)=\{\theta \in \Theta: \rho (q_{\theta}(S_{i_1}(t,y_1)),q_{\theta}(S_{i_2}(t,y_2)))\leq q\rho_c(x_1,x_2)\}.$$
Observe that $R_1(t)\subset R_2(t)$ for $t\in T$. To see this, let $t\in T$ and  $\theta \in R_1(t)$. From \hyperref[cnd:a3]{(A3)} \eqref{st3:m}, \eqref{st4:ea} and \eqref{nc} it then follows that
\begin{align*}
\begin{split}
\rho (q_{\theta}(S_{i_1}(t,y_1)),q_{\theta}(S_{i_2}(t,y_2)))&\leq L_q\rho (S_{i_1}(t,y_1),S_{i_2}(t,y_2))\\
&\leq  L_q L e^{\alpha t}\rho(y_1,y_2)+L_{q}M_{\mathcal{L}}e^{\sup T}\delta(i_1,i_2)\\
&\leq q\rho(y_1,y_2)+ \frac{\olambda L L_q}{\ulambda-\alpha} \frac{(\ulambda-\alpha)M_{\mathcal{L}}}{\olambda L}  e^{\sup T}\delta(i_1,i_2)\\
&\leq q\rho(y_1,y_2)+qc\delta(i_1,i_2)=q\rho _c (x_1,x_2),
\end{split}
\end{align*}
which gives the desired inclusion. Since \hbox{$R_2(t)=\{\theta\in\Theta:\, \boldsymbol{q}_j(x_1,x_2,t,\theta,t)\in U\}$}, we therefore obtain 
$$\mathbbm{1}_U(\textbf{q}_j(x_1,x_2,t,\theta))=\mathbbm{1}_{R_2(t)}(\theta)\geq \mathbbm{1}_{R_1(t)}(\theta)\;\;\;\mbox{for}\;\;\;t\in T,\,\theta\in\Theta.$$
Hence, applying the fact that $R_1(t)=\theta(S_{i_1}(t,y_1),S_{i_2}(t,y_2))$ and \hyperref[cnd:a6]{(A6)}, we can conclude that
\begin{align*}
\begin{split}
\Gamma _t(x_1,x_2,U)&=\sum_{j\in I} \int_{\Theta}\mathbbm{1}_U(\textbf{q}_j(x_1,x_2,t,\theta))\boldsymbol{\pi}_j(x_1,x_2,t,\theta )\textbf{p}_{\theta}(x_1,x_2,t)\,d\theta\\
&\geq  \delta_{\pi}\int_{R_1(t)}\textbf{p}_{\theta}(x_1,x_2,t)d\theta =\delta_{\pi}\delta _p.
\end{split}
\end{align*}
Consequently, using \eqref{cnd:a7}, we infer that
\begin{align*}
\begin{split}
Q(x_1,x_2,U)&\geq \int_T\boldsymbol{\lambda}(x_1,x_2,t)\boldsymbol{L}(x_1,x_2,t)\Gamma_t(x_1,x_2,U)dt\\
&\geq \delta _{\pi}\delta _p\int_T\ulambda e^{-\olambda t}dt=\delta,
\end{split}
\end{align*}
which proves \hyperref[cnd:B4]{(B4)}.

\textbf{Step 5.}
What is left is to show that  \hyperref[cnd:B5]{(B5)} holds. For this purpose, let $(x_1,x_2):=((y_1,i_1),(y_2,i_2))\in F$, and define 
$$z_1(t):=S_{i_1}(t,y_1)\;\;\;\mbox{and}\;\;\;z_2(t):=S_{i_2}(t,y_2)\;\;\;\mbox{for}\;\;\;t\geq 0.$$ 
Applying the following inequality:
\begin{equation}
\label{estmin}
(u_1\wedge u_2)(v_1\wedge v_2)\geq u_1v_1-u_1|v_1-v_2|-v_1|u_1-u_2|\;\;\;(u_i,v_i\in\mathbb{R}),
\end{equation}
and keeping in mind that
$$\boldsymbol{\pi}_j(x_1,x_2,t,\theta )=\pi _{i_1j}(q_{\theta}(z_1(t)))\wedge \pi _{i_2j}(q_{\theta}(z_2(t))),$$
$$\textbf{p}_{\theta}(x_1,x_2,t)=p_{\theta}(z_1(t))\wedge p_{\theta}(z_2(t)),$$
we obtain
\begin{align}
\label{eq1:s5}
\begin{split}
\Gamma_t(x_1,x_2,X^2) &=\sum_{j\in I} \int_{\Theta}\boldsymbol{\pi}_j(x_1,x_2,t,\theta )\textbf{p}_{\theta}(x_1,x_2,t)d\theta \\
&\geq \sum_{j\in I}\int_{\Theta} \pi_{i_1j}(q_{\theta}(z_1(t)))p_{\theta}(z_1(t))\,d\theta \\
&\,\quad-\sum_{j\in I}\int_{\Theta} \pi_{i_1j}(q_{\theta}(z_1(t))) |p_{\theta}(z_1(t))-p_{\theta}(z_2(t))|\,d\theta\\
&\,\quad-\sum_{j\in I} \int_{\Theta} |\pi_{i_1j}(q_{\theta}(z_1(t)))-\pi_{i_2j}(q_{\theta}(z_2(t)))| p_{\theta}(z_1(t))\,d\theta.
\end{split}
\end{align}
Let $C_k(t)$ (where $k\in\{1,2,3\}$) denote the $k$th sum
on the right-hand side of the above inequality. Clearly, $C_1(t)=1$ for any $t\geq 0$. By condition \hyperref[cnd:a5]{(A5)} we have
$$C_2(t)=\int_{\Theta} |p_{\theta}(z_1(t))-p_{\theta}(z_2(t))|\,d\theta\leq L_p\rho(z_1(t),z_2(t)).$$ Further, from \hyperref[cnd:a5]{(A5)} and \hyperref[cnd:a3]{(A3)} it follows that
\begin{align*}
C_3(t)&\leq \int_{\Theta} \Big[\sum_{j\in I} |\pi_{i_1j}(q_{\theta}(z_1(t)))-\pi_{i_1j}(q_{\theta}(z_2(t)))| \Big]p_{\theta}(z_1(t))\,d\theta\\
&\,\quad+\delta(i_1,i_2)\int_{\Theta} \Big[\sum_{j\in I} |\pi_{i_1j}(q_{\theta}(z_2(t)))-\pi_{i_2j}(q_{\theta}(z_2(t)))| \Big]p_{\theta}(z_1(t))\,d\theta\\
&\leq L_{\pi} \int_{\Theta}\rho(q_{\theta}(z_1(t)),q_{\theta}(z_2(t)))p_{\theta}(z_1(t))\,d\theta+2\delta(i_1,i_2)\\
&\leq L_qL_{\pi}\rho(z_1(t),z_2(t))+2\delta(i_1,i_2).
\end{align*}
Consequently, from \eqref{eq1:s5} we can now conclude that
\begin{align*}
\Gamma_t(x_1,x_2,X^2)&\geq C_1(t)-C_2(t)-C_3(t)\\
&\geq 1-(L_p+L_q L_{\pi})\rho(S_{i_1}(t,y),S_{i_2}(t,y))-2\delta(i_1,i_2)\\
&\geq 1-(L_p+L_q L_{\pi}+1)\left[\rho(S_{i_1}(t,y),S_{i_2}(t,y))+2\delta(i_1,i_2)\right],
\end{align*}
which, together with \eqref{st3:m}, gives
\begin{equation}
\label{a5-a}
\Gamma_t(x_1,x_2,X^2)\geq 1-L(L_p+L_q L_{\pi}+1)\left[e^{\alpha t}\rho(y_1,y_2)+\frac{tM_{\mathcal{L}}+2}{L}\delta(i_1,i_2)\right].
\end{equation}

Let us recall that
$$\boldsymbol{\lambda}(x_1,x_2,t)=\lambda(z_1(t))\wedge\lambda(z_2(t)),$$
$$\boldsymbol{L}(x_1,x_2,t)=e^{-\int_0^t \lambda(z_1(s))\,ds} \wedge e^{-\int_0^t \lambda(z_2(s))\,ds}.$$
We now apply \eqref{estmin} again to see that
\begin{align}
\begin{split}
\label{a5-b}
\int_0^{\infty} \boldsymbol{\lambda}(x_1,&x_2,t)\boldsymbol{L}(x_1,x_2,t)\,dt\geq \int_0^{\infty} \lambda(z_1(t))e^{-\int_0^t \lambda(z_1(s))\,ds}\,dt\\
&\quad-\int_0^{\infty} \lambda(z_1(t))\left|e^{-\int_0^t \lambda(z_1(s))\,ds}-e^{-\int_0^t \lambda(z_2(s))\,ds}\right|\,dt\\
&\quad-\int_0^{\infty} e^{-\int_0^t \lambda(z_1(s))\,ds} |\lambda(z_1(t))-\lambda(z_2(t))|\,dt.
\end{split}
\end{align}
Let $I_k$ (where $k\in\{1,2,3\}$) stand for the $k$th integral
on the right-hand side of the above inequality. Since the integrand of $I_1$ is the density of the distribution \eqref{tdist}, we have $I_1=1$. In order to estimate $I_2$ we use the inequality
$$\left|e^u-e^v\right|\leq e^{\max(u,v)}|u-v|,\;\;\;u,v\in\mathbb{R}.$$
It follows that
$$
I_2\leq \int_0^{\infty} \lambda (z_1(t))e^{\max\left(-\int_0^t\lambda(z_1(s))\,ds,\;-\int_0^t\lambda(z_1(s))\,ds \right)} \left[\int_0^t |\lambda(z_1(s))-\lambda(z_2(s))|\,ds\right]dt.
$$
According to \eqref{cnd:a7} and \hyperref[cnd:a4]{(A4)} we now obtain
$$
I_2\leq \olambda L_{\lambda}\int_0^{\infty} e^{-\ulambda t} \left[\int_0^t \rho(S_{i_1}(s,y_1),S_{i_2}(s,y_2))\,ds\right]dt.\\
$$
Finally, using \hyperref[cnd:a2]{(A2)} (cf. also \eqref{st3:m}) we infer that
\begin{align*}
I_2&\leq \olambda L_{\lambda}\int_0^{\infty} e^{-\ulambda t} \left[\int_0^t 
(Le^{\alpha s}\rho(y_1,y_2)+M_{\mathcal{L}}s\,\delta(i_1,i_2))\,ds\right]\,dt\\
&\leq \olambda L_{\lambda}\int_0^{\infty} e^{-\ulambda t}\left[ \frac{L}{\alpha}(e^{\alpha t}-1)\rho(y_1,y_2)+\frac{M_{\mathcal{L}}}{2}t^2\delta(i_1,i_2)\right]dt\\
&=\olambda L_{\lambda} \left[\frac{L}{\alpha}\int_0^{\infty} (e^{(\alpha-\ulambda)t}-e^{-\ulambda t})\,dt\, \rho(y_1,y_2)+ \frac{M_{\mathcal{L}}}{2} \int_0^{\infty} t^2 e^{-\ulambda t}\,dt\,\delta(i_1,i_2)\right]\\
&=\olambda L_{\lambda}\left [\frac{L}{\ulambda(\ulambda-\alpha)}\rho(y_1,y_2)+\frac{M_{\mathcal{L}}}{\ulambda^3} \delta(i_1,i_2)\right]\\
&=\frac{\olambda LL_{\lambda}}{\ulambda(\ulambda-\alpha)}\left[\rho(y_1,y_2)+\frac{M_{\mathcal{L}}(\ulambda-\alpha)}{L\ulambda^2}\delta(i_1,i_2) \right],
\end{align*}
which, in accordance with \eqref{nc}, gives
\begin{equation}
\label{a5-b1}
I_2\leq \frac{\olambda LL_{\lambda}}{\ulambda(\ulambda-\alpha)}\rho_c(x_1,x_2).
\end{equation}
Conditions \eqref{cnd:a7}, \hyperref[cnd:a4]{(A4)} and \hyperref[cnd:a2]{(A2)} also enable us to estimate $I_3$ as follows:
\begin{align*}
I_3&\leq L_{\lambda} \int_0^{\infty}e^{-\ulambda t} \rho(S_{i_1}(t,y_1),S_{i_2}(t,y_2))\,dt\\
&\leq L_{\lambda} \int_0^{\infty} e^{-\ulambda t} (Le^{\alpha t}\rho(y_1,y_2)+M_{\mathcal{L}} t\rho(i_1,i_2))\,dt \\
&\leq L_{\lambda} \left[\frac{L}{\ulambda-\alpha}\rho(y_1,y_2)+\frac{M_{\mathcal{L}}}{\ulambda^2}\delta(i_1,i_2) \right]\\
&=\frac{LL_{\lambda}}{\ulambda-\alpha}\left[\rho(y_1,y_2)+\frac{M_{\mathcal{L}}(\ulambda-\alpha)}{L\ulambda^2}\delta(i_1,i_2) \right].
\end{align*}
Consequently, by \eqref{nc} we then have
\begin{equation}
\label{a5-b2}
I_3\leq \frac{LL_{\lambda}}{\ulambda-\alpha}\rho_c(x_1,x_2)\leq  \frac{\olambda LL_{\lambda}}{\ulambda(\ulambda-\alpha)}\rho_c(x_1,x_2).
\end{equation}
From \eqref{a5-b}, \eqref{a5-b1} and \eqref{a5-b2} it now follows that
$$\int_0^{\infty} \boldsymbol{\lambda}(x_1,x_2,t)\boldsymbol{L}(x_1,x_2,t)\,dt\geq 1-\frac{2\olambda LL_{\lambda}}{\ulambda(\ulambda-\alpha)}\rho_c(x_1,x_2).$$

Combining \eqref{a5-a} with the latter inequality and using \eqref{cnd:a7} again, we can now deduce that
\begin{align*}
Q(x_1,x_2,&X^2)=\int_0^{\infty} \boldsymbol{\lambda}(x_1,x_2,t)\boldsymbol{L}(x_1,x_2,t)\Gamma_t(x_1,x_2,X^2)\,dt\\
&\geq 1-\frac{2\olambda LL_{\lambda}}{\ulambda(\ulambda-\alpha)}\rho_c(x_1,x_2)\\
&\quad -  \int_0^{\infty}\olambda\, e^{-\ulambda t} L(L_p+L_q L_{\pi}+1)\left[e^{\alpha t}\rho(y_1,y_2)+\frac{tM_{\mathcal{L}}+2}{L}\delta(i_1,i_2)\right]\,dt.
\end{align*}
Evaluating the last term on the right-hand side of this estimation, say $I_0$, we see that
\begin{align*}
I_0&=\olambda L(L_p+L_q L_{\pi}+1)\left[\frac{1}{\ulambda-\alpha}\rho(y_1,y_2)+\left(\frac{M_{\mathcal{L}}}{L\ulambda^2}+\frac{2}{L\ulambda} \right)\delta(i_1,i_2)\right]\\
&=\frac{\olambda L(L_p+L_q L_{\pi}+1)}{\ulambda-\alpha}\left[\rho(y_1,y_2)+ \left(\frac{(\ulambda-\alpha)M_{\mathcal{L}}}{L\ulambda^2}+\frac{2(\ulambda-\alpha)}{L\ulambda} \right)\delta(i_1,i_2)\right]\\
&\geq \frac{\olambda L(L_p+L_q L_{\pi}+1)}{\ulambda-\alpha}\rho_c(x_1,x_2),
\end{align*}
where the last inequality is due to \eqref{nc}. Hence, finally, we obtain
\begin{align*}
Q(x_1,x_2,X^2)\geq 1-\left(\frac{2\olambda LL_{\lambda}}{\ulambda(\ulambda-\alpha)}+ \frac{L(L_p+L_q L_{\pi}+1)}{\ulambda-\alpha}\right)\rho_c(x_1,x_2).
\end{align*}
This establishes \hyperref[cnd:B5]{(B5)} and completes the proof of the theorem.
\end{proof}

\section{Application to a Poisson-driven stochastic differential equation} \label{sec:4}
In this part of the paper, we consider a PDSDE in the spirit of Lasota and Traple \cite{b:las_poiss}. In the case discussed here we assume that the intensity of stochastic perturbations (jumps) depends on the solution (like in \cite{b:kazak}), and that the unperturbed part of the equation is governed by a finite collection of randomly switched dynamical systems $y'(t)=a(y(t),i)$, $i\in\{1,\ldots,N\}$ (as in \cite{b:horbacz_poiss}). 

We shall focus on the Markov operator corresponding to the change of distributions of the solution process from jump to jump (that is, the jump operator). Theorem \ref{thm_main} will be used to provide sufficient conditions ensuring the exponential ergodicity of such an \hbox{operator.}

\subsection{Poisson random measure and Poisson point process}
Let us first introduce notation and recall some basic concepts (adapted mainly from \cite[\S1.7-1.9]{b:situ}) concerning Poisson random measures, which will be needed in the rest of the paper.

Suppose we are given a measurable space $(S,\Sigma_S)$, and let $(\Omega, \mathcal{F},\pr)$ be a probability space. Recall that a map $\m:\Sigma_S\times\Omega\to\left[0,\infty\right]$ is called a \emph{random measure} if, for any $A\in\Sigma_S$, $\m(A,\cdot):\Omega\to \left[0,\infty\right]$ is a random variable, and for any $\omega\in\Omega$, $\m(\cdot,\omega):\Sigma_S\to\left[0,\infty\right]$ is a $\sigma$-finite measure. In what follows, we sometimes identify $\m$ with the map $\overline{\m}:\Omega\to \overline{\M}(S)$ given by $\overline{\m}(\omega)(A):=\m(\omega, A)$ for $\omega\in\Omega$ and $A\in\Sigma_S$, where $\overline{\M}(S)$ denotes the set of all non-negative $\sigma$-finite measures on~$\Sigma_S$.

A random measure $\m:\Sigma_S\times\Omega\to\left[0,\infty\right]$ is said to be a \emph{Poisson random measure} with intensity $\lambda_\m:\Sigma_S\to\left[0,\infty\right]$ whenever
\begin{itemize}
\item[(i)] for each $A\in\Sigma_S$, the random variable $\m(A,\cdot)$ is Poisson distributed with mean $\ew \left[\m(A,\cdot)\right]=\lambda_\m(A)$, i.e. $\pr(\m(A,\cdot)=k)=(\lambda_\m(A)^k/k!)e^{-\lambda_\m(A)}$ for every $k\in\n_0$;
\item[(ii)] if $A_1,\ldots,A_n\in\Sigma_S$ are disjoint sets then $\m(A_1,\cdot),\ldots,\m(A_n,\cdot)$ are \hbox{mutually} independent.
\end{itemize}
In the above definition, we adopt the convention that $0\cdot\infty:=0$.
Thus, if \hbox{$\lambda_\m(A)=\infty$}, then $\pr(\m(A,\cdot)=k)=0$ for all $k\in\mathbb{N}_0$, whence $\m(A,\cdot)=0$ almost everywhere.

Let us now consider a measurable space $(\Theta, \Sigma_{\Theta})$, and define $S_{\Theta}:=\mathbb{R}_+\times\Theta$, where $\mathbb{R}_+:=\left[0,\infty\right)$. We endow $S_{\Theta}$ with the product $\sigma$-field $\Sigma_{S_{\Theta}}:=\mathcal{B}(\mathbb{R}_+)\otimes\Sigma_{\Theta}$.

A mapping $p:D_p\to \Theta$ is called a \emph{point function} valued on $\Theta$ whenever $D_p$ is a countable subset of $(0,\infty)$. Let $\Pi(\Theta)$ denote the set of all point functions valued on $\Theta$. Every $p\in \Pi(\Theta)$ defines a counting measure \hbox{$N_p: \Sigma_{S_{\Theta}} \to \n_0\cup\{\infty\}$} specified by
\begin{align*}
N_p(\left[0,t\right]\times A)&:=\card\{s\in D_p:\,s\leq t,\; p(s)\in A\}\\
&=\sum_{s\in D_p} \mathbbm{1}_{\left[0,t\right]\times A}(s,p(s))\;\;\;\mbox{for}\;\;\; t\geq 0,\,A\in\Sigma_{\Theta}.
\end{align*}

For any given function $\p:\Omega\to \Pi(\Theta)$, let us now define $\Np:\Sigma_{S_{\Theta}}\times \Omega \to\n\cup\{\infty\}$ by $$\Np(B,\omega):=N_{\p(\omega)}(B)\;\;\;\mbox{for}\;\;\;B\in\Sigma_{S_{\Theta}},\,\omega\in\Omega,$$ which can as well be viewed as the map $\Np:\Omega \to \overline{\M}(S_{\Theta})$ determined by $\Np(\omega):=N_{\p(\omega)}$ for $\omega\in\Omega$. To simplify notation, for $t\geq 0$ and $A\in\Sigma_\Theta$, we will often write $\Np(t,A)$ instead of $\Np(\left[0,t\right]\times A)$.

A~map $\p:\Omega\to \Pi(\Theta)$ is said to be a \emph{(Poisson) point process} if $\Np$ is a (Poisson) random measure. In this case, $\Np$ is called a \emph{(Poisson) random counting measure}. 

For a Poisson point process $\p$, by its intensity we mean the intensity of the Poisson random measure $\Np$, i.e. $n_{\p}(B):=\ew\left[\Np(B)\right]$ for $B\in\Sigma_{S_{\Theta}}$. If $n_{\p}$ satisfies
$$n_{\p}(\left[0,t\right]\times A)=t\kappa(A)\;\;\;\mbox{for}\;\;\;t>0,\,A\in  \Sigma_{\Theta},$$
where $\kappa$ is some non-negative measure on $\Sigma_{\Theta}$, then $\p$ is called a \emph{stationary Poisson process}, and $\kappa$ is said to be the \emph{characteristic measure} of $\p$.

Let us now quote \cite[Corollary 55]{b:situ} together with a useful statement extracted from the proof of \cite[Theorem 54]{b:situ} (cf. also \cite[\S8-9]{b:ikeda} and \cite{b:kazak}).
\begin{thm}\label{thm:ex_poiss}
Let $\kappa:\Sigma_{\Theta}\to\left[0,\infty\right]$ be a $\sigma$-finite measure. Then, on some probability space $(\Omega,\mathcal{F},\pr)$, there exists a stationary Poisson point process \hbox{$\p:\Omega\to\Pi(\Theta)$} with the characteristic measure $\kappa$. In the case where $\kappa$ is a finite measure, the appropriate $\p$ can be defined so that, for any $\omega\in\Omega$, \hbox{$\p(\omega):D_{\p(\omega)}\to\Theta$} is given by
$$\p(\omega)(\btau_n(\omega)):=\eta_n(\omega),\;\;\;D_{\p(\omega)}=\{\btau_n(\omega):n\in\n\},$$
where
\begin{itemize}
\item[(i)] $\btau_n:\Omega\to \left[0,\infty\right)$, $n\in\n$, forms a strictly increasing sequence of random variables with $\tau_n\to\infty$, whose increments $\Delta\btau_{n}:=\btau_{n}-\btau_{n-1}$, where $\btau_0:=0$, are mutually independent and have the same exponential distribution with rate $\kappa(\Theta)$;
\item[(ii)] $\eta_n:\Omega\to\Theta$, $n\in\n$, forms a sequence of   mutually independent and identically distributed random variables with the common distribution $\kappa/\kappa(\Theta)$, such that the sequences $(\eta_n)_{n\in\n}$ and $(\btau_n)_{n\in\n}$ are independent.
\end{itemize}
In particular, the Poisson random counting measure corresponding to $\p$ takes then the form
\begin{equation}
\label{count_form}
\Np(t,A)=\sum_{n=1}^{\infty} \mathbbm{1}_{\{\btau_n\leq t,\,\eta_n\in A\}}\;\;\;\mbox{for}\;\;\;t\geq 0,\,A\in\Sigma_{\Theta},
\end{equation}
and $\Np(t,A)<\infty$ a.s. for any $t\geq 0$ and $A\in \Sigma_{\Theta}$.
\end{thm}
If a random counting measure $\Np$ has the form \eqref{count_form}, then, for any given $A\in\Sigma_{\Theta}$, the variables $\overline{\tau}_n$ are called \emph{jump times} of $(\Np(t,A))_{t\geq 0}$. 

Suppose now that we are given a Banach space $(H,\norma{\cdot})$ and a point process \hbox{$\p:\Omega\to\Pi(\Theta)$}. Assume that $(\Omega,\mathcal{F},\pr)$ is equipped with a filtration $\{\mathcal{F}_t(\p)\}_{t\geq 0}\subset\mathcal{F}$ such that $\Np(t,A)$ is $\mathcal{F}_t(\p)$-measurable for every $A\in \Sigma_{\Theta}$. Further, let $\mathcal{G}(\p)$ denote the family of all functions $F:\left[0,\infty\right)\times\Theta\times\Omega\to H$ such that
\begin{itemize}
\item[(a)] for each $(\theta,\omega)\in \Theta\times\Omega$, the map $t\mapsto F(t,\theta,\omega)$ is right-continuous,
\item[(b)] for each $t>0$, the map $\Theta\times\Omega \ni (\theta,\omega)\mapsto F(t,\theta,\omega)\in Y$ is $\mathcal{B}(H)\times \mathcal{F}_t(\p)$-measurable.
\end{itemize}
Given any $F\in \mathcal{G}(\p)$ and $t>0$ such that
$$\sum_{\{s\in D_{\p(\omega)}:\,s\leq t\}} \norma{F(s,\p(\omega)(s),\omega)}<\infty\;\;\;\mbox{for almost all} \;\;\;\omega\in\Omega,$$
we can define the integral of $F$ with respect to $\Np$ by setting
$$\int_0^t \int_{\Theta} F(s,\theta,\cdot)\,\Np(ds,d\theta):=\sum_{\{s\in D_{\p(\cdot)}:\,s\leq t\}} F(s,\p(\cdot)(s),\cdot)\;\;\;\mbox{a.e.}$$
Clearly, $I_t$ is a random variable, due to (b).

Finally, let us define the integral with respect to $\Np(\Lambda(ds),d\theta)$, where $(\Lambda(t))_{t\geq 0}$ is a real-valued stochastic process with strictly increasing trajectories, such that $\Lambda(0)=0$ and $D_{\p(\omega)}\subset \{\Lambda(t)(\omega):\,t>0\}$ for all $\omega\in\Omega.$ For this purpose, consider the map $\p_{\Lambda}:\Omega\to\Pi(\Theta)$ with $\p_{\Lambda}(\omega):D_{\p_{\Lambda}(\omega)}\to\Theta$ given by
\begin{equation}
\label{p_lambda}
\p_{\Lambda}(\omega)(\widehat{s}):=\p(\omega)(\Lambda(\widehat{s})(\omega)),\;\;\;\widehat{s}\in D_{\p_{\Lambda}(\omega)}:=\bigcup_{s\in D_{\p(\omega)}}\{\widehat{s}>0:\,\Lambda(\widehat{s})(\omega)=s\}
\end{equation}
for every $\omega\in\Omega$. It then follows that
\begin{align*}
\Np(\Lambda(t),A)&=\card\{s\in D_{\p}:\,s\leq \Lambda(t),\;\p(s)\in A\}\\
&=\card\{\widehat{s}\in D_{\p_{\Lambda}}:\,\Lambda(\widehat{s})\leq \Lambda(t),\;\p(\Lambda(\widehat{s}))\in A\}\\
&=\card\{\widehat{s}\in D_{\p_{\Lambda}}:\, \widehat{s}\leq t,\; \p_{\Lambda}(\widehat{s})\in A\}=\Npw(t,A),
\end{align*}
and thus, for $F\in\mathcal{G}(\p_{\Lambda})$, it is natural to define
\begin{align}
\label{defint}
\int_0^t \int_{\Theta} F(s,\theta,\cdot)\,\Np(\Lambda(ds),d\theta)&:=\sum_{\{\widehat{s}\in D_{\p_{\Lambda}(\cdot)}:\,\widehat{s}\leq t\}} F(\widehat{s},\p_{\Lambda}(\cdot)(\widehat{s}),\cdot)\;\;\;\mbox{a.e.}
\end{align}

\subsection{Model description and assumptions} \label{sub:p2}
We can now give a formal description for the aforementioned model. Let $(H,\<\cdot|\cdot\>)$ be a separable Hilbert space endowed with the norm $\norma{\cdot}$ induced by the inner product $\<\cdot|\cdot\>$, and let $Y$ be a non-empty closed subset of~$H$. We assume that $Y$ is endowed with the metric generated by $\norma{\cdot}$. Further, consider a finite set $I=\{1,\ldots,N\}$ and a matrix of continuous functions $\pi_{ij}:Y\to\left[0,1\right]$, $i,j\in I$, satisfying \eqref{stoch_pi}. Finally, let $\Theta$ be an arbitrary (and fixed) compact interval, and define the probability measure $\kappa:\B(\Theta)\to \left[0,1\right]$ of the form
$$\kappa(A):=\int_A h(\theta)\,d\theta,\;\;A\in \mathcal{B}(\Theta),$$ where $h:\Theta\to\left[0,\infty\right)$ is a continuous probability density function.

Moreover, assume that we are given three maps $\sigma:Y\times\Theta\to Y$, $\lambda: Y\to\left(0,\infty\right)$ and $a:Y\times I\to Y$ such that  the following statements hold:
\begin{itemize}
\phantomsection
\label{cond:P0}
\item[(P0)] $\sigma$ is continuous and $\lambda$ satisfies \eqref{cnd:a7};
\label{cond:P1}
\item[(P1)] There exists $y_*\in Y$ such that
$$\int_{\Theta}\norma{\sigma(y_*,\theta)}\,\kappa(d\theta)<\infty.$$
\phantomsection
\label{cond:P2}
\item[(P2)] There exists $L_{\sigma}>0$ such that
$$\int_{\Theta}\norma{\sigma(y_1,\theta)-\sigma(y_2,\theta)}\kappa(d\theta)\leq L_{\sigma}\norma{y_1-y_2}\;\;\;\mbox{for}\;\;\; y_1,y_2\in Y.$$

\phantomsection
\label{cond:P3}
\item[(P3)] The maps $a(\cdot,i)$, $i\in I$, are bounded on bounded sets and satisfy the following conditions:
\begin{itemize}
\item[(P3.1)] There exists a (negative) constant $\alpha<\ulambda-(1+L_{\sigma})\olambda$ such that each $a(\cdot,i)$ is \hbox{$\alpha$-dissipative}, that is, for every $i\in I$ and any $y_1,y_2\in Y,$
$$\<a(y_1,i)-a(y_2,i)\,|\,y_1-y_2\>\leq \alpha\norma{y_1-y_2}^2;$$
\item[(P3.2)] There exists $T>0$ such that $Y\subset \range(\id_Y-t a(\cdot,i))$ for all $t\in(0,T)$ and every $i\in I$.
\end{itemize}

\phantomsection
\label{cond:P4}
\item[(P4)] There exists $L_{\lambda}>0$ such that
$$|\lambda(y_1)-\lambda(y_2)|\leq L_{\lambda} \norma{y_1-y_2}\;\;\;\mbox{for}\;\;\;y_1,y_2\in Y.$$

\phantomsection
\label{cond:P5}
\item[(P5)]
There exists $L_{\pi}>0$ such that
$$
\sum_{j\in I}|\pi_{ij}(y_1)-\pi_{ij}(y_2)|\leq L_{\pi}\norma{y_1-y_2} \;\;\;\textrm{for}\;\;\; y_1, y_2 \in Y,\,i\in I.\\
$$

\phantomsection
\label{cond:P6}
\item[(P6)] There exist $\delta _{\pi}>0$ and $\delta _h>0$ such that
\begin{equation} 
\begin{split}
& \sum_{j\in I}\min\{\pi_{i_1j}(y_1),  \pi_{i_2j}(y_2)  \}\geq \delta_{\pi} \;\;\;\textrm{for}\;\;\; y_1, y_2  \in Y,\; i_1,i_2\in I;\\
&\kappa (\{\theta\in\Theta:\, \norma{\sigma(y_1,\theta)-\sigma(y_2,\theta)}\leq \norma{y_1-y_2}\})\geq \delta_h \;\;\;\textrm{for}\;\;\; y_1, y_2  \in Y.
\end{split}
\end{equation}
\end{itemize}

Let us now consider the stochastic differential equation
\begin{equation}
\label{SDE1}
dY(t)=a(Y(t),\xi(t))\,dt+\int_{\Theta}\sigma(Y(t),\theta)\,\Np(\Lambda(dt),d\theta)\\
\end{equation}
with intial condition \noeqref{SDE2}
\begin{equation}
\label{SDE2}
Y(0)=Y_0
\end{equation}
for an unknown process  $\{Y(t)\}_{t\geq 0}$ with values in $Y$, where 
\begin{align}
\begin{split}
\label{SDE3}
&\Lambda(t)=\int_0^t \lambda(Y(s))\,ds,\\
&\xi(t)=\xi_n\;\;\;\mbox{if}\;\;\;\Np(\Lambda(t),\Theta)=n,\;\;n\in\n_0,
\end{split}
\end{align}
for $t\geq 0$, and $Y_0$, $\xi_0$, $\p$ and $\{\xi_n\}_{n\in\n}$ are defined on a suitable probability space $(\Omega,\mathcal{F},\pr)$ as follows:
\begin{itemize} 
\item[\sbullet] $\xi_0:\Omega \to I$ and $Y_0:\Omega\to Y$ are random variables with an arbitrary (and fixed) distributions;
\item[\sbullet] $\p:\Omega\to \Pi(\Theta)$ is a stationary Poisson process with the characteristic measure $\kappa$. According to Theorem \ref{thm:ex_poiss} we can assume that $\p$ is determined by two sequences $\{\btau_n\}_{n\in\n_0}$ and $\{\eta_n\}_{n\in\n}$ of random variables satisfying conditions (i) and (ii) (given in that theorem), in the sense that
\begin{equation}
\label{jump_btau}
\p(\omega)(\btau_n(\omega))=\eta_n(\omega)\;\;\;\mbox{for}\;\;\;\omega\in\Omega.
\end{equation}
In particular, $\{\eta_n\}_{n\in\n}$ is then a sequence of $\Theta$-valued mutually independent random variables with the same density $h$;
\item[\sbullet] $\xi_0:\Omega\to I$ is a random variable with  an arbitrary (and fixed) distribution, and $\{\xi_n\}_{n\in\n}$ is a sequence of $I$-valued random variables defined so that
$$\pr(\xi_{n}=j\;|\;Y(\tau_n)=y,\, \xi_{n-1}=i)=\pi_{ij}(y)\;\;\;\mbox{for}\;\;\;i, j\in I,\;y\in Y,$$ where $\tau_n:\Omega\to \left[0,\infty\right)$, $n\in\n$, are the jump times of $\Npw$, determined by \eqref{p_lambda}, that is
\begin{equation}
\label{def_tau}
\Lambda(\tau_n)=\btau_n \;\;\;\mbox{for}\;\;\;n\in\n_0.
\end{equation}
\end{itemize}

By a solution of \eqref{SDE1}-\eqref{SDE3} we mean a c\`adl\`ag process $\{Y(t)\}_{t\geq 0}$, taking values in $Y$, such that
\begin{equation}
\label{sol_def}
Y(t)=Y_0+\int_0^t a(Y(s),\xi(s))\,ds+\int_0^t\int_{\Theta}\sigma(Y(s-),\theta)\,\Np(\Lambda(ds),d\theta),
\end{equation}
where $\{\Lambda(t)\}_{t\geq 0}$ and $\{\xi(t)\}_{t\geq 0}$ are determined by \eqref{SDE3}. Clearly, due to \eqref{def_tau}, $\{\xi(t)\}_{t\geq 0}$ can be equivalently written as
\begin{equation}
\xi(t)=\xi_n\;\;\;\mbox{for}\;\;\;t\in\left[\tau_n,\tau_{n+1}\right), \;\;\;n\in\n_0.
\end{equation}

Having in mind the definition of $\p_{\Lambda}$, given in \eqref{p_lambda}, and applying \eqref{jump_btau} and \eqref{def_tau}, we see that
$$D_{\p_{\Lambda}(\omega)}=\{\tau_n(\omega):\,n\in\n\},$$
$$\p_{\Lambda}(\omega)(\tau_n(\omega))=\p(\omega)(\btau_n(\omega))=\eta_n(\omega)\;\;\;\mbox{for}\;\;\;\omega\in\Omega,\,n\in\n.$$
Consequently, using \eqref{defint} for $F(s,\theta,\omega):=\sigma(Y(s-)(\omega),\theta)$, we obtain
\begin{align}
\label{int_jumps}
\begin{split}
\int_0^t\int_{\Theta}\sigma(Y(s-),\theta)\,\Np(\Lambda(ds),d\theta)
&=\sum_{\{\widehat{s}\in D_{\p_{\Lambda}}:\,\,\widehat{s}\leq t\}} \sigma(Y(\widehat{s}-),\p_{\Lambda}(\widehat{s}))\\
&=\sum_{\{n\in\n:\,\,\tau_n\leq t\}} \sigma(Y(\tau_n-),\eta_n).
\end{split}
\end{align}

For each $i\in I$, let us now consider the Cauchy problem of the form
\begin{equation} \label{df_v} v'(t)=a(v(t),i),\,\;\;\;v(0)=y\;\;\;\mbox{where}\;\;\;y\in Y.\end{equation}
From condition \hyperref[cond:P3]{(P3)}, \cite[Corollary 5.4]{b:kazufami} and \cite[Theorem 5.11]{b:kazufami} it follows that there exists a semiflow $S_i:\mathbb{R}_+\times I\to\mathbb{R}$ satisfying

\begin{equation}
\label{S_prop1}
\norma{S_i(t,y_1)-S_i(t,y_2)}\leq e^{\alpha t}\norma{y_1-y_2}\;\;\;\mbox{for}\;\;\;y_1,y_2\in Y,
\end{equation}
\begin{equation}
\label{S_prop2}
\norma{S_i(t,y)-y}\leq t\norma{a(y,i)}\;\;\;\mbox{for}\;\;\;y\in Y,
\end{equation}
such that, for any $y\in Y$, the map $t\mapsto S_i(t,y)$ is the unique solution {of \eqref{df_v}}. 
%Clearly, conditions \eqref{S_prop1} and \eqref{S_prop2} guarantee that $\{S_i:\,i\in I\}$ satisfy \hyperref[cnd:a2]{(A2)} with $\mathcal{L}$ given by $\mathcal{L}(y):=2\max_{i\in I} \norma{a(y,i)}$, which is bounded on bounded sets, as required.

We will show that the solution of \eqref{SDE1}-\eqref{SDE3} is given by
\begin{equation}
\label{sol_SDE1}
Y(t):=S_{\xi_n}(t-\tau_n,Y(\tau_n))\;\;\;\mbox{for}\;\;\;t\in\left[\tau_n,\tau_{n+1}\right),
\end{equation}
where
\begin{equation}
\label{sol_SDE2}
Y(\tau_n):=Y(\tau_n-)+\sigma(Y(\tau_n-),\eta_n).
\end{equation}
For this purpose, let us denote the right-hand side of \eqref{sol_def} by $U(t)$, i.e. 

\begin{equation}
U(t):=Y_0+\int_0^{t} a(Y(s),\xi(s))\,ds+\int_0^{t}\int_{\Theta}\sigma(Y(s-),\theta)\,\Np(\Lambda(ds),d\theta).
\end{equation}
We first observe that $U(\tau_n)=Y(\tau_n)$ for any $n\in\n$. To see this, suppose that such an equality holds for an arbitrary, but fixed $n$. Applying \eqref{int_jumps} and the fact that $U(\tau_n)=Y(\tau_n)$, we obtain
\begin{align*}
&U(\tau_{n+1})=Y(\tau_n)+\int_{\tau_n}^{\tau_{n+1}} a(Y(s),\xi(s))\,ds+\sigma(Y(\tau_{n+1}-),\eta_{n+1}).
\end{align*}
The substitution $u=s-\tau_n$ gives
\begin{align*}
\int_{\tau_n}^{\tau_{n+1}} a(Y(s),\xi(s))&\,ds=\int_0^{\Delta \tau_{n+1}} a(S_{\xi_n}(u,Y(\tau_n)),\xi_n)\,du\\
&=\int_0^{\Delta \tau_{n+1}} \frac{d}{du} S_{\xi_n}(u,Y(\tau_n))\,du\\
&=S_{\xi_n}(\Delta \tau_{n+1},Y(\tau_n))-S_{\xi_n}(0,Y(\tau_n))=Y(\tau_{n+1}-)-Y(\tau_n),
\end{align*}
which implies that
$$U(\tau_{n+1})=Y(\tau_{n+1}-)+\sigma(Y(\tau_{n+1}-),\eta_{n+1})=Y(\tau_{n+1}).$$
Now, letting $n\in\n$ and $t\in\left[\tau_n,\tau_{n+1}\right)$, we can conclude that
\begin{align*}
U(t)&=U(\tau_n)+\int_{\tau_n}^t a(Y(s),\xi(s))\,ds=Y(\tau_n)+\int_{\tau_n}^t a(S_{\xi_n}(s-\tau_n, Y(\tau_n)),\xi_n)\,ds\\
&=Y(\tau_n)+\int_0^{t-\tau_n} a(S_{\xi_n}(u, Y(\tau_n)),\xi_n)\,du\\
&=Y(\tau_n)+\int_0^{t-\tau_n} \frac{d}{du}S_{\xi_n}(u, Y(\tau_n))\,du=S(t-\tau_n,Y(\tau_n))=Y(t),
\end{align*}
where the first equality follows from \eqref{int_jumps}.

\subsection{Exponential ergodicity of the jump operator associated with the PDSDE}
Let $\{Y(t)\}_{t\geq 0}$ be the solution of \eqref{SDE1}-\eqref{SDE3} specified by \eqref{sol_SDE1}.  We are concerned with the sequence of random variables given by the post-jump locations of the process $\{(Y(t),\xi(t))\}_{t\geq 0}$, that is, $\{(Y_n,\xi_n)\}_{n\in\n_0}$, wherein \hbox{$Y_n:=Y(\tau_n)$} is determined by \eqref{sol_SDE2}. If we define  
\hbox{$q_{\theta}: Y\to Y$ by}
$$q_{\theta}(y):=y+\sigma(y,\theta)\;\;\;\mbox{for}\;\;\;y\in Y,\,\theta\in \Theta,$$
then, due to \eqref{sol_SDE1} and \eqref{sol_SDE2}, we can write
$$Y_n=q_{\eta_n}(Y(\tau_n-))=q_{\eta_n}(S_{\xi_{n-1}}(\Delta \tau_n,Y_{n-1}))\;\;\;\mbox{for}\;\;\;n\in\n.$$

There is no loss of generality in assuming that $Y_0$, $\xi_0$, $\{\btau_n\}_{n\in\n}$, $\{\eta_n\}_{n\in\n}$ and $\{\xi_n\}_{n\in\n}$ satisfy the independence conditions detailed in Section \ref{sec:3}. In that case $\{(Y_n,\xi_n)\}_{n\in\n_0}$ is a time-homogeneus Markov chain with values in $X:=Y\times I$, whose transition law has the form  \eqref{def:stP} with $p_{\theta}(\cdot)\equiv h(\theta)$, $\theta\in\Theta$. To see this, it suffices to  show that the conditional distribution of $\tau_n$ is determined by \eqref{tdist}. For this purpose, define $H:\mathbb{R}_+\times Y\times I\to\mathbb{R}_+$ so that $H(\cdot,y,i)$ is the inverse of $L(\cdot,y,i)$ for every $(y,i)\in Y\times I$. Then
\begin{equation}
\label{eh}
\Delta \tau_{n+1}=H(\Delta\btau_{n+1},Y_n,\xi_n)\;\;\;\mbox{for}\;\;\;n\in\n_0,
\end{equation}
since
\begin{align*}
L(\Delta\tau_{n+1},Y_n,\xi_n)&=\int_0^{\Delta\tau_{n+1}}\lambda (S_{\xi_n}(s,Y_n))\,ds=\int_{\tau_n}^{\tau_{n+1}} \lambda (S_{\xi_n}(u-\tau_n,Y_n))\,du\\
&=\int_{\tau_n}^{\tau_{n+1}} \lambda (Y(u))\,du=\Lambda(\tau_{n+1})-\Lambda(\tau_n)=\Delta\btau_{n+1},
\end{align*}
where the last equality follows from \eqref{def_tau}. Using this, we obtain
\begin{align*}
\pr(\Delta \tau_{n+1}\leq t|Y_n=i,\,\xi_n=y)&=\pr(H(\Delta\btau_{n+1},Y_n,\xi_n)\leq t|Y_n=i,\,\xi_n=y)\\
&=\pr(\Delta\btau_{n+1}\leq L(t,y,i))=1-e^{-L(t,y,i)},
\end{align*}
which is the desired conclusion.

It is now straightforward to establish the exponential ergodicity of the Markov chain $\{(Y_n,\xi_n)\}_{n\in\n_0}$ by the use of Theorem \ref{thm_main}.
%The exponential ergodicity of the Markov chain $\{(Y_n,\xi_n)\}_{n\in\n_0}$ is now a simple consequence of Theorem \ref{thm_main}. 
\begin{thm}
Suppose that the functions  $\sigma:Y\times\Theta\to Y$, $\lambda: Y\to\left(0,\infty\right)$ and \hbox{$a:Y\times I\to Y$} satisfy conditions \hyperref[cond:P0]{(P0)}-\hyperref[cond:P6]{(P6)}, and let $\{(Y_n,\xi_n)\}_{n\in\n_0}$ be the Markov chain given by the post-jump locations of the process $\{(Y(t),\xi(t))\}_{t\geq 0}$ specified by \eqref{SDE1}-\eqref{SDE3}. Further, let $P$ be the Markov operator corresponding to $\{(Y_n,\xi_n)\}_{n\in\n_0}$. Then, for a sufficiently large $c$, the operator $P$ has a unique invariant probability measure $\mu_*\in\Mp$, which is exponentially attracting. More precisely, $\mu_*\in\Mp^{\rho_c,1}(X)$ and there exists $x_*\in X$, $C\in\mathbb{R}$ and $\beta\in\left[0,1\right)$ such that
$$
\norma{\mu P^n-\mu_*}_{FM}\leq C \beta^n(\int_X\rho_c(x_*,x)(\mu+\mu_*)(dx)+1)
$$
for any $\mu\in\Mp^{\rho_c,1}(X)$ and any $n\in\n$, where $\rho_c$ is given by \eqref{def:roc}.
\end{thm}

\begin{proof}
In view of Theorem \ref{thm_main}, it suffices to show that conditions \hyperref[cnd:a1]{(A1)}-\hyperref[cnd:a6]{(A6)} hold with $p(\theta,\cdot)\equiv h(\theta)$ for $\theta\in\Theta$, and $L$, $L_q$, $\ulambda$, $\olambda$, $\alpha$ satisfying \eqref{cnd:LL}. 

First of all, as we have mentioned in Section \ref{sub:p2}, condition \hyperref[cond:P3]{(P3)} ensures that the semiflows $S_i$, $i\in I$, generated by \eqref{df_v}, enjoys properties \eqref{S_prop1} and \eqref{S_prop2}. This clearly implies that \hyperref[cnd:a2]{(A2)} holds with $L=1$, $\alpha<\ulambda-(1+L_{\sigma})\olambda$ and $\mathcal{L}$ given by $\mathcal{L}(y):=2\max_{i\in I} \norma{a(y,i)}$, which is bounded on bounded sets, as required.

Further, we show that condition \hyperref[cnd:a1]{(A1)} is satisfied. From \hyperref[cond:P1]{(P1)} we know that 
$$M:=\int_{\Theta} \norma{\sigma(y_*,\theta)}h(\theta)\,d\theta<\infty\;\;\;\mbox{for some}\;\;\;y_*\in Y.$$
Keeping in mind \eqref{S_prop2} and applying \hyperref[cond:P2]{(P2)} we have
\begin{align*}
\int_{\Theta} \|q_{\theta}(&S_i(t,y_*))-y_*\|h(\theta)\,d\theta
=\int_{\Theta} \norma{S_i(t,y_*)+\sigma(S_i(t,y_*),\theta)-y_*}h(\theta)\,dt\theta\\
&\leq \int_{\Theta} \norma{S_i(t,y_*)-y_*}h(\theta)\,d\theta+ \int_{\Theta} \norma{\sigma(S_i(t,y_*),\theta)-\sigma(y_*,\theta)}h(\theta)\,d\theta\\
&\quad+\int_{\Theta} \norma{\sigma(y_*,\theta)}h(\theta)\,d\theta\leq t\norma{a(y_*,i)}+L_{\sigma}t\norma{a(y_*,i)}+M\\
&\leq (1+L_{\sigma})\max_{j\in I} \norma{a(y_*,i)}t+M\;\;\;\mbox{for all}\;\;\;i\in I.
\end{align*}
Hence, setting $K:=(1+L_{\sigma})\max_{j\in I} \norma{a(y_*,i)}$, we obtain
\begin{align*}
\int_0^{\infty} e^{-\ulambda t}\int_{\Theta} \|q_{\theta}(&S_i(t,y_*))-y_*\|h(\theta)\,d\theta\,dt\leq \frac{K}{\ulambda^2}+\frac{M}{\ulambda}<\infty\;\;\;\mbox{for all}\;\;\;i\in I.
\end{align*}
From hypothesis \hyperref[cond:P2]{(P2)} it follows directly that \hyperref[cnd:a3]{(A3)} holds with $L_q:=1+L_{\sigma}$, since
\begin{align*}
\int_{\Theta} \norma{q_{\theta}(y_1)-q_{\theta}(y_2)}h(\theta)\,d\theta
&\leq \norma{y_1-y_2}+\int_{\Theta} \norma{\sigma(y_1,\theta)-\sigma(y_2,\theta)}\kappa(d\theta)\\
&\leq (1+L_{\sigma})\norma{y_1-y_2}.
\end{align*}
Conditions \hyperref[cnd:a4]{(A4)} and \hyperref[cnd:a5]{(A5)} are just equivalent to \hyperref[cond:P4]{(P4)} and \hyperref[cond:P5]{(P5)}, respectively. Moreover, \hyperref[cnd:a6]{(A6)} gives immediately \hyperref[cond:P6]{(P6)}, since
$$\{\theta\in \Theta:\norma{\sigma(y_1,\theta)-\sigma(y_2,\theta)}\leq L_{\sigma}\norma{y_1-y_2}\}$$
is a subset of
$$\{\theta:\norma{q_{\theta}(y_1)-q_{\theta}(y_2)}\leq L_{q}\norma{y_1-y_2}\}.$$

Finally, using the upper bound of $\alpha$, specified in \hyperref[cond:P3]{(P3.1)}, we infer that
$$LL_q\olambda+\alpha= (1+L_{\sigma})\olambda+\alpha< \ulambda,$$
which finishes the proof.
\end{proof}
\bibliographystyle{acm}
\bibliography{references}
\end{document}